\documentclass[a4paper,11pt,reqno]{amsart}

\usepackage[english]{babel}
\usepackage{amsmath}
\usepackage{amsfonts}
\usepackage{amsthm}
\usepackage{float}
\usepackage{stmaryrd}
\usepackage{graphics}
\usepackage{graphicx}
\usepackage{subfig}
\usepackage{datetime}
\usepackage{bm}
\usepackage[colorlinks=false]{hyperref}
\usepackage{color}

\newcommand{\mc}{\mathcal}
\newcommand{\eps}{\varepsilon}
\renewcommand{\d}{\,\mathrm{d}}

\DeclareMathOperator*{\argmin}{argmin}
\DeclareMathOperator{\sgn}{sgn}

\def\enne{\mathbb{N}}

\def\erre{\mathbb{R}}

\renewcommand{\to}{\rightarrow}

\numberwithin{equation}{section}
\newtheorem{thm}{Theorem}[section]
\newtheorem{defi}[thm]{Definition}
\newtheorem{prop}[thm]{Proposition}
\newtheorem{lemma}[thm]{Lemma}
\newtheorem{cor}[thm]{Corollary}

\theoremstyle{definition}
\begingroup
\newtheorem{rmk}[thm]{Remark}

\endgroup

\theoremstyle{remark}

\begin{document}
	\author[E. Bonetti, C. Cavaterra, F. Freddi and F. Riva]{Elena Bonetti, Cecilia Cavaterra, Francesco Freddi and Filippo Riva}
	
	\title[Damage in laminates with cohesive interface]{On a phase--field model of damage \\ for hybrid laminates with cohesive interface}
	
	\begin{abstract}
		In this paper we investigate a rate--independent model for hybrid laminates described by a damage phase--field approach on two layers
		coupled with a cohesive law governing the behaviour of their interface in a one--dimensional setup. For the analysis we adopt the notion of energetic evolution,
		based on global minimisation of the involved energy. Due to the presence of the cohesive zone, as already emerged in literature, compactness issues lead to the introduction of a fictitious variable replacing the physical one which represents the maximal opening of the
		interface displacement discontinuity reached during the evolution. A new strategy which allows to recover the equivalence between the fictitious
		and the real variable under general loading--unloading regimes is illustrated. The argument is based on temporal regularity of energetic evolutions.
		This regularity is achieved by means of a careful balance between the convexity of the elastic energy of the layers and the natural concavity of the
		cohesive energy of the interface.
	\end{abstract}
	
	\maketitle
	
	{\small
		\par
	}

	\pagenumbering{arabic}
	
	\medskip
	
	\tableofcontents

	\section*{Introduction}
	Composite fibre reinforced materials are increasingly finding applications in the manufacturing industry due to their capacity of offering high strength and stiffness
	with low mass density. Their only mechanical weakness is the brittleness. Indeed, rapid failure occurs without sufficient warning, due to the intrinsic nature of the
	adopted materials. A possible strategy to provide a ductile failure response is to consider novel composite architectures where fibres of different stiffness and ultimate
	strain values are combined through cohesive interfaces (hybridisation). In this case, complex rupture processes occur with diffuse crack pattern (fragmentation) and/or
	delamination. A deep analytical comprehension of the failure mechanisms of these kind of materials is thus needed in order to predict and control the appearance and the
	evolution of the cracks. \par
	Among the mathematical community, the variational approach to fracture, as formulated by \cite{BourFrancMar,FrancMar}, is one of the most adopted viewpoints to deal with
	crack problems. It is based on the Griffith's idea \cite{Griffith} that the crack growth is governed by a reciprocal competition between the internal elastic energy of the
	body and the energy spent to increase the crack length. In the original theory the energy associated with the fracture is proportional to the measure of the fracture itself,
	while in the cohesive case (Barenblatt \cite{Barenblatt}), where the process is more gradual, the energy depends on the opening of the crack.\par
	Due to the complexity of the phenomenon and the technical difficulties of the related mathematical analysis, especially from the numerical point of view, in the last twenty
	years a phase-field damage approach has been developed to overcome the aforementioned issues. Nowadays it is a well established and consolidated method to approximate both
	brittle (see \cite{AmbrTort1,AmbrTort2, Giacomini}) and cohesive fractures (see \cite{BonConIurl,Iurl}). It consists in the introduction of an irreversible
	internal variable taking values in $[0,1]$ and representing the damage state of the material. Usually, values $0$ and $1$ mean a completely sound and
	a completely broken state, respectively, while a value in between represents the case of a partial damage. The presence of a fracture is thus ideally replaced by those parts
	of the body whose damage variable has reached the value $1$.\par
	In this work a rigorous mathematical analysis is carried out for a one--dimensional model for hybrid laminates, which was previously introduced and numerically investigated in \cite{AleFredd1d}. Its
	description is given by coupling the damage phase--field approach, which models the elastic--brittle behaviour of the layers, with a cohesive law in the interface
	connecting the materials. The investigation is restricted to the case of incomplete damage in the sense that a reservoir of elastic material stiffness is always maintained, even if the damage variable reaches the maximum value $1$.
	This situation can be concretely justified by considering materials formed by different components from which only a part can undergo a damage (for instance in composite materials obtained with a matrix and a reinforcement) and delamination may take place; on the other side it can be seen as a mathematical approximation of the complete damage setting in which the material goes through full rupture. We refer for instance to \cite{BouchMielkRoub,MielkRoub} for an analysis of complete damage between two viscoelastic bodies, or to \cite{BonFreSeg} for a complete damage model in elastic materials, while we postpone the inspection of this model to future works, due to high mathematical difficulties related to the cohesive zone.\par 		
	Here, the model we want to analyse describes the evolution of a unidirectional hybrid laminate in hard device: a prescribed time--dependent displacement $\bar u(t)$ is applied
	on one side of the bar, whereas the other is fixed. We restrict our attention to slow prescribed displacements, so that inertial effects can be neglected and the analysis can
	be included in a quasi--static and rate--independent regime. For the sake of simplicity we consider a bar composed by only two layers with thickness $\rho_1$ and $\rho_2$,
	respectively, bonded together along the entire length by a cohesive interface. The thickness of the interface is very thin compared with $\rho_1$ and $\rho_2$, which in turn
	are way smaller than the length of the laminate $L>0$. Thus the model can be considered as one--dimensional.\par
	As already mentioned, the brittle behaviour of the two elastic layers is modelled by a phase--field damage approach. It suits with the rate--independent framework we are considering. For the reader
	interested instead in dynamic and rate--dependent damage models we refer for instance to \cite{BonSchim, Fremond}. The unknowns that govern the problem are thus the displacements
	of the two layers, denoted by $u_1$ and $u_2$, and their irreversible damage variables $\alpha_1$ and $\alpha_2$.\par
	Despite its apparent simplicity, the model has considerable application potential for the study and analysis of different failure phenomena in thin multilayered materials subjected to membranal mechanical regime such as composite materials. In \cite{AleFredd2d, AleFredd1d} it has been adopted to investigate the complex failure modes of hybrid laminates. The experimental evidences have been successfully replicated in 1D and 2D settings. The model, suitably extended to anisotropic materials and/or curvilinear geometries, can be an extremely powerful tool to analyse craquelure phenomena in artworks as preliminarily highlighted in \cite{Negri}.\par 
	In the quasi--static setting a huge variety of notions of solution can be considered, see for instance the monograph \cite{MielkRoubbook}. In this paper we focus our attention
	on the concept of energetic evolution, based on two ingredients: at every time the solution is a global minimiser of the involved total energy, and the sum of internal and
	dissipated energy balances the work done by the external prescribed displacement. The same kind of evolution in an analogous cohesive fracture model between two elastic bodies
	is studied in \cite{CagnToad, DMZan}; other notions based on stationary points of the energy, always in the framework of cohesive fractures, are instead analysed in
	\cite{NegSca, NegVit}.\par
	The choice of working with energetic evolutions is motivated by the future aim of analysing the complete damage situation, for which the main tool usually adopted (see \cite{BouchMielkRoub, MielkRoub}) is given by $\Gamma$-convergence \cite{DalMasoGamma}, notion which fits well with global minimisers.\par
	The total energy we consider is composed by a first part taking into account elastic responses of the layers and dissipation due to damage, and a second part reflecting
	the cohesive behaviour of the interface. The cohesive interface is governed by the slip between the two layers $\delta=|u_1-u_2|$ and its irreversible counterpart $\delta_h$
	which represents the maximal slip achieved during the evolution. The presence of an irreversible history variable can be also found in different models than cohesive
	fracture: we mention for instance the notion of fatigue, investigated in \cite{AleCrismOrl, CrisLazzOrl}.
	
	The expression of the energy in the model under consideration is hence given by:
	\begin{equation*}
		\sum_{i=1}^{2}\rho_i\Bigg(\underbrace{\frac {1}{2}\int_{0}^{L}\!\!\!\!E_i(\alpha_i(x))(u_i'(x))^2 \d x}_\text{elastic energy of the i-th layer}
		+\!\!\underbrace{\frac {1}{2}\int_{0}^{L} \!\!\!(\alpha'_i(x))^2 \d x}_{\substack{\text{internal energy of}\\ \text{the i-th damage variable}}}
		+\!\!\!\!\underbrace{\int_{0}^{L}\!\!\!\!w_i(\alpha_i(x))\d x}_{\substack{\text{energy dissipated by}\\\text{damage in the i-th layer}}}\!\!\!\!\Bigg)
		+\underbrace{\int_{0}^{L}\!\!\!\!\varphi(\delta(x),\delta_h(x))\d x}_{\substack{\text{ internal and dissipated}\\\text{energy in the interface}}},
	\end{equation*}
	where the symbol prime $'$ denotes the one--dimensional spatial derivative, $E_i\colon[0,1]\to (0,+\infty)$ is the elastic Young modulus of the $i$-th layer (which is strictly positive since we are in the incomplete damage framework),
	$w_i\colon [0,1]\to [0,+\infty)$ is a dissipation density and $\varphi\colon \{(y,z)\in\erre^2\mid z\ge y\ge 0\}\to [0,+\infty)$ is the loading-unloading density of
	the cohesive interface.\par
	As usual in the context of energetic evolutions, we follow a time--discretisation algorithm to show existence of solutions. More precisely, we consider a fine partition
	of the time interval $[0,T]$ and at each time step we select a global minimiser of the total energy; we then recover the time-continuous evolution by sending to zero the
	discretisation parameter. Due to compactness issues regarding the maximal slip $\delta_h$, the time--discretisation process leads to
	the introduction of a weaker notion of solution where a fictitious history variable $\gamma$ replaces the concrete one $\delta_h$. We point out that this auxiliary variable only appears when dealing with global minima of the energy, indeed it can be found in \cite{CagnToad,DMZan}, but not in \cite{NegSca, NegVit} where stationary points
	are considered. The issue has been partially overcome in \cite{CagnToad,DMZan} with different approaches, but assuming the hypothesis of constant unloading response, namely
	when the loading--unloading density $\varphi$ depends only on the second variable $z$.\par
	Here, an original strategy based on temporal regularity properties of energetic evolutions in order to recover the equivalence between the fictitious variable $\gamma$ and
	the proper one $\delta_h$ under reasonable assumptions on the density $\varphi$ is illustrated. In particular, we are able to cover all the general cases considered in
	\cite{NegSca}. Moreover, the proposed approach fits well with the model under consideration, but it can be also adapted to more general situations.\par
	An alternative strategy to deal with cohesive problems can be found in literature, where adhesion is treated with the introduction of a damage variable that
	macroscopically defines the bond state between two solids. Detachment corresponds to full damage state. The problem has been investigated theoretically in
	\cite{BonBonfRos1,BonBonfRos2,BonBonfRos3,BonBonfRos4} and numerically in \cite{FreFre,FreIur}
	\par
	The paper is organised as follows. In Section~\ref{secsetting} we introduce in a rigorous way the variational problem, presenting the global and history variables:
	the displacement field $u_i$, the damage variables $\alpha_i$, the slip $\delta$ and the history slip $\delta_h$. Subsequently, details of the involved energies are given
	together with a precise notion of energetic evolution and of its weak counterpart, here named generalised energetic evolution, including the fictitious variable $\gamma$.\par
	Section~\ref{secexistence} is devoted to the proof of existence of generalised energetic evolutions under very mild assumptions on the loading--unloading cohesive density
	$\varphi$. We first introduce the time--discretisation algorithm based on global minimisation of the energy, and we provide uniform bounds on the sequence of discrete
	minimisers. Thanks to these bounds and by means of a suitable version of Helly's selection theorem we are able to extract convergent subsequences as the time step vanishes.
	After the introduction of the fictitious history variable $\gamma$ and by exploiting the fact that the discrete functions selected by the algorithm are global minima of the
	total energy, we finally deduce that the previously obtained limit functions actually are a generalised energetic evolution.\par
	In Section~\ref{seceulerlagrange} attention is focused on the equations that a generalised energetic evolution must satisfy; they are a byproduct of the global minimality
	condition together with the energy balance. It turns out that the displacements fulfil a system of equations in divergence form, see \eqref{systemstresses}, while the damage variables
	satisfy a Karush--Kuhn--Tucker condition, see \eqref{KKT}, assuming a priori certain regularity in time. Of course these equations have to be meant in a weak sense.
	The results of this third section are a first step in order to obtain the equivalence between $\gamma$ and the concrete history variable $\delta_h$.\par
	Finally Section~\ref{sectemporalregularity} illustrates the main result of the paper. We first adapt a convexity argument introduced in \cite{Thom} to our setting in which
	a cohesive energy (concave by nature) is present, in order to gain regularity in time (absolute continuity) of generalised energetic evolutions. Once this temporal regularity
	is achieved, we exploit the Euler--Lagrange equations of Section~\ref{seceulerlagrange} together with the monotonicity (in time) of $\gamma$ and $\delta_h$ to deduce their
	equivalence under reasonable assumptions on $\varphi$. We thus obtain as a byproduct that the generalised energetic evolution found in Section~\ref{secexistence} is actually
	an energetic evolution, since $\gamma$ coincides with $\delta_h$.\par
	At the end of the work we attach an Appendix in which we gather some definitions and properties  we need throughout the paper about absolutely continuous and
	bounded variation functions with values in Banach spaces.
	
	\section{Setting of the Problem}\label{secsetting}	In this section we present the variational formulation of the one-dimensional continuum model described in the Introduction
	of two layers bonded together by a cohesive interface in a hard device setup.  We list all the main assumptions we need throughout the paper. We also introduce the
	two notions of energetic evolution and generalised energetic evolution in our context, see Definitions~\ref{Enersol} and \ref{Generalisedenersol}.\par
	For the sake of clarity, in this work every function in the Sobolev space $H^1(a,b)$ is always identified with its continuous representative. The prime symbol $'$ is used to denote spatial derivatives, while the dot symbol $\,\dot{}\,$ to denote time derivatives. In the case of a function $f\colon [0,T]\to H^1(a,b)$, which thus depends on both time and space, we write $f(t)'$ to denote the (weak) spatial derivative of $f(t)\in H^1(a,b)$ and with a little abuse of notation we write $f'(t,x)$ to denote its value at a.e. $x\in [a,b]$. If $f$ is sufficiently regular in time, for instance in $C^1([0,T];H^1(a,b))$, for the time derivative we instead adopt the scripts $\dot f$, $\dot{f}(t)$ and $\dot{f}(t,x)$, with the obvious meanings: $\dot f$ is the function from $[0,T]$ to $ H^1(a,b)$, $\dot{f}(t)$ is its value as a function in $H^1(a,b)$, once $t\in [0,T]$ is fixed, and $\dot{f}(t,x)$ is its value (as a real number) at $x\in [a,b]$.  By $a\vee b$ and
	$a\wedge b$ we finally mean the maximum and the minimum between two extended real numbers $a$ and $b$ in $[-\infty,+\infty]$.\par
	We fix a time $T>0$ and the length of the laminate $L>0$. We also normalise the thickness of the two layers $\rho_1$ and $\rho_2$ to $1$, since this does not affect the results.
	\subsection{The variables}To describe the evolution of the system, for $i=1,2$ we introduce the function $u_i\colon[0,T]\times[0,L]\to \erre$, where $u_i(t,x)$ denotes the
	displacement at time $t$ of the point $x$ of the $i$-th layer; here $\bm{u}(t,x)$ represents the vector in $\erre^2$ with components $u_1(t,x)$ and $u_2(t,x)$. For
	the structure of the model itself, at every time $t\in [0,T]$ the displacement $u_i(t)$ will belong to the space $H^1(0,L)$.
	The function $\delta\colon [0,T]\times[0,L]\to [0,+\infty)$ defined as
	\begin{subequations}
		\begin{equation}
			\delta(t,x)=\delta[\bm{u}](t,x):=|u_1(t,x)-u_2(t,x)|,
		\end{equation}
		instead denotes the displacement slip on the interface between the two layers. Then, we  introduce the non-decreasing function
		$\delta_h\colon [0,T]\times[0,L]\to [0,+\infty)$ as
		\begin{equation}\label{historicalslip}
			\delta_h(t,x):=\sup\limits_{\tau\in[0,t]}\delta(\tau,x),
		\end{equation}
	\end{subequations}
	namely the history variable which records the maximal slip reached at the point $x$ in the interface till the time $t$. Internal constraints, such as unilateral conditions
	(see \cite{BonBonfRos2,BonBonfRos3}), are not necessary on the kinematics as this only permits displacement slips between the two solids and interpenetration is prevented a-priori.
	
	Finally, for $i=1,2$, we consider the function $\alpha_i\colon [0,T]\times[0,L]\to [0,1]$, where $\alpha_i(t,x)$ represents the amount of damage at time $t$ of the point $x$ of the
	$i$-th layer. It is non-decreasing in time with values in $[0,1]$. The value $0$ means completely sound material whereas the value $1$ represents fully damaged state. We however point out that we confine ourselves to the incomplete damage setting, namely the fully damaged state does not describe the rupture of the layer, whose stiffness indeed never vanishes; this will be clear in \eqref{Ecoerc}, in which we assume a strictly positive elastic modulus for both layers. As for the displacement, the damage
	variable $\alpha_i(t)$ will be in $H^1(0,L)$ for every $t\in [0,T]$. In analogy with the previous setting, $\bm{\alpha}(t,x)$ denotes the vector in $\erre^2$ with components
	$\alpha_1(t,x)$ and $\alpha_2(t,x)$.\par

	\subsection{The energies} We now present the energies involved in our model. Given a pair $(\bm{u},\bm{\alpha})$ belonging to $[H^1(0,L)]^2\times[H^1(0,L)]^2$ and representing
	an admissible displacement and damage, the stored elastic energy of the two layers is given by:
	\begin{equation}
		\mathcal{E}[\bm{u},\bm{\alpha}]:=\sum_{i=1}^{2}\frac {1}{2}\int_{0}^{L}E_i(\alpha_i(x))(u_i'(x))^2 \d x,
	\end{equation}
	where, for $i=1,2$, we assume the elastic Young moduli $E_i$ satisfy:
	\begin{equation}\label{Ecoerc}
		E_i\in C^0([0,1])\text{ such that }E_i(y)\geq \min_{\tilde y \in [0,1]} E_i(\tilde y)=: \eps_i >0, \, \text{ for every }y\in[0,1].
	\end{equation}
	We define
	\begin{equation}\label{epsilon}
		\eps:=\eps_1\wedge\eps_2>0,
	\end{equation}
	which is strictly positive by \eqref{Ecoerc}. This feature reflects the fact that we are considering the incomplete damage framework, and it will be used to gain coercivity of
	$\mc E$. This property of the energy is indeed missing in the complete damage setting where the functions $E_i$ can vanish, and a completely different notion of solution and
	strategy must be adopted. We refer to \cite{BouchMielkRoub, MielkRoub} or to \cite{BonFreSeg} for the interested reader. \par
	We can now introduce for $i=1,2$ the stress $\sigma_i\colon [0,T]\times[0,L]\to \erre$, defined as
	\begin{equation}\label{stresses}
		\sigma_i(t,x)=\sigma_i[u_i,\alpha_i](t,x):= E_i(\alpha_i(t,x))u_i'(t,x).
	\end{equation}
	As before, by $\bm{\sigma}(t,x)$ we mean the vector with components $\sigma_1(t,x)$ and $\sigma_2(t,x)$.\par
	Another energy term appearing in the model is the sum of the stored and the dissipated energy of the phase--field variable $\bm{\alpha}\in [H^1(0,L)]^2$ during the damaging
	process and expressed by
	\begin{equation}	\mathcal{D}[\bm{\alpha}]:=\sum_{i=1}^{2}\left(\frac {1}{2}\int_{0}^{L} (\alpha'_i(x))^2 \d x+\int_{0}^{L}w_i(\alpha_i(x))\d x\right).
	\end{equation}
	In literature there are very different choices of dissipation
	functions $w_i$ (see for instance \cite{AleFredd2d,AleFredd1d,Kuhn,Nejar,PhamMarigo,Wu}). As elementary examples we can consider $w_i(y)=\frac{y^2+y}{2}$ or $w_i(y)=y$.
	
	In this work we permit quite general assumptions on $w_i$ as follows:
	\begin{equation}\label{wi}
		w_i\in C^0([0,1])\text{ such that }w_i(y)\ge 0\text{ for every }y\in[0,1].
	\end{equation}
	\begin{rmk}
		The dissipated damage density, usually a process dependent function (i.e. depending on the time derivative of the damage variable $\dot{\alpha}(t)$), is here treated as a
		state function due to the underlying gradient damage model. See also \cite{AleFredd2d}, \cite{AleFredd1d} and \cite{MielkRoub}.
	\end{rmk}
	We finally introduce the cohesive energy in the interface between the two layers:
	\begin{equation}\label{Cohesiveenergy}
		\mathcal{K}[\delta,\gamma]:=\int_{0}^{L}\varphi(\delta(x),\gamma(x))\d x,
	\end{equation}
	where $\delta$ and $\gamma$ are two non-negative functions in $[0,L]$ such that $\gamma\ge\delta$ and representing, respectively, the slip and the history slip of the
	displacement at a given instant. The non-negative function
	\begin{equation}
		\varphi:\mc T\to [0,+\infty),\quad\text{where}\quad\mc T=\{(y,z)\in \erre^2\mid z\ge y\ge 0\},
	\end{equation}
	is the loading-unloading density of the cohesive interface; the variable $y$ governs the unloading regime (usually convex), while $z$ the loading regime
	(usually concave).\par
	Since several assumptions on $\varphi$ will be needed throughout the paper we prefer listing them here. The first set of assumptions, very mild, will be used
	in Section~\ref{secexistence} to prove existence of (generalised) energetic evolutions (see Definitions~\ref{Enersol} and \ref{Generalisedenersol}):
	\begin{itemize}
		\item[($\varphi$1)] $\varphi$ is lower semicontinuous;
		\item[($\varphi$2)] $\varphi(0,\cdot)$ is bounded in $[0,+\infty)$;
		\item[($\varphi$3)] $\varphi(y,\cdot)$ is continuous and non-decreasing in $[y,+\infty)$, for every $y\ge 0$.
	\end{itemize}
	We also present here the specific -- but often not suited for physical applications -- assumption which has been used in \cite{CagnToad} and \cite{DMZan} to deal with the fictitious variable $\gamma$
	(see Definition~\ref{Generalisedenersol}), and which in this work we are able to avoid. We however include it in the list because we make use of it
	in Theorem~\ref{exensol}, where we employ the argument of \cite{CagnToad} in our context:
	\begin{itemize}
		\item[($\varphi$4)] there exist two functions  $\varphi_1,\varphi_2\colon [0,+\infty)\to[0,+\infty)$ such that $\varphi_1$ is lower semicontinuous,
		$\varphi_2$ is bounded, non-decreasing and concave, and  $\varphi(y,z)=\varphi_1(y)+\varphi_2(z)$.
	\end{itemize}
	We notice that ($\varphi$4) implies ($\varphi$1), ($\varphi$2) and ($\varphi$3).
	\begin{rmk}
		Actually, in \cite{CagnToad, DMZan} the function $\varphi_1$ appearing above is chosen identically $0$, so that the cohesive density $\varphi$ depends only on the
		second variable $z$ (constant unloading regime). However, their argument can be easily adapted to our case.
	\end{rmk}
	To overcome the necessity of ($\varphi$4) in recovering the equality $\gamma=\delta_h$ (see and compare Definitions~\ref{Enersol} and \ref{Generalisedenersol}), in Sections~\ref{seceulerlagrange} and \ref{sectemporalregularity} we develop an alternative and new argument based on time regularity of solutions (we however point out that condition ($\varphi$4) is completely unrelated with temporal regularity and in general does not imply it).
	The assumptions we need to perform the whole strategy are listed just below. For the sake of brevity, given a non-negative function $\varphi$ with domain
	$\mc T$, for $z\in[0,+\infty)$ we define $$\psi(z):=\varphi(z,z),$$ namely the restriction of $\varphi$ on the diagonal. The function $\psi$ governs the loading regime. Moreover we introduce the constant
	\begin{equation}\label{bardelta}
		\bar\delta:=\inf\{z>0\mid \psi \text{ is constant in }[z,+\infty)\},
	\end{equation}
	with the convention $\inf \{\emptyset\}=+\infty$; it represents the limit slip which triggers complete delamination. Indeed, according to \cite{AleFredd2d,AleFredd1d}, complete delamination may occour for finite or infinite slip value (see Remark \ref{rmkexamplephi}).
	
	We then set
	\begin{equation*}
		\mc T_{\bar\delta}:=\{(y,z)\in\mc T\mid z< \bar\delta\}.
	\end{equation*}
	We thus require:
	\begin{itemize}
		\item[($\varphi$5)] the function $\psi\in C^1([0,+\infty))$ is $\lambda$--convex for some $\lambda>0$, namely for every $\theta\in [0,1]$ and $z^a,z^b\in [0,+\infty)$ it holds
		\begin{equation*}
			\psi(\theta z^a+(1-\theta)z^b)\le \theta\psi (z^a)+(1-\theta)\psi (z^b)+\frac \lambda 2\theta(1-\theta)|z^a-z^b|^2;
		\end{equation*}
		\item[($\varphi$6)] for every $z\in(0,+\infty)$ the map $\varphi(\cdot,z)\in C^1([0,z])$ is non-decreasing and convex;
		\item[($\varphi$7)] for every $z\in(0,+\infty)$ there holds $\partial_y\varphi(z,z)=\psi'(z)$ and $\partial_y\varphi(0,z)=0$;
		\item[($\varphi$8)] the partial derivative $\partial_y\varphi$ belongs to $C^0(\mc T\setminus(0,0))$ and it is bounded in $\mc T$.
		\item[($\varphi$9)] for every $y\in [0,\bar{\delta})$ the map $\varphi(y,\cdot)$ is differentiable in $[y,\bar\delta)$ and the partial derivative $\partial_z\varphi$ is 
		continuous  and strictly positive on $\mc T_{\bar\delta}\setminus\{(z,z)\in\erre^2\mid z\ge 0\}$.
	\end{itemize}
	Condition ($\varphi$9) will be actually weakened in Section~\ref{sectemporalregularity}, where only a uniform strict monotonicity with respect to $z$ will be needed,
	see \eqref{bilipass}.\par
	We want to point out that this set of assumptions includes a huge variety of mechanically meaningful loading--unloading densities $\varphi$, as precised in the next remark.
	We also notice that these conditions are similar to the one considered in \cite{NegSca}.
	
	\begin{rmk}[\textbf{Main Example}]\label{rmkexamplephi}
		The prototypical example of a physically meaningful loading-unloading density is obtained reasoning in the opposite way of what we presented before, namely firstly a function
		$\psi$ is given and then the density $\varphi$ is built from $\psi$. As regards $\psi$, which governs the loading regime, natural assumptions arising from applications are the following:
		$\psi\in C^2([0,\bar\delta))\cap C^1([0,+\infty))$ is a non-decreasing, concave and bounded function such that $\psi(0)=0$, $\psi'>0$ and $\psi''$ is bounded from below 
		in $[0,\bar\delta)$. In particular ($\varphi$5) is satisfied with $\lambda=\sup\limits_{z\in[0,\bar\delta)}|\psi''(z)|$. For instance one can consider:
		\begin{equation*}
			\psi(z)=\begin{cases}
				cz(2k-z),&\text{ if }z\in[0,k),\\
				ck^2,&\text{ if }z\in[k,+\infty),
			\end{cases}\quad\quad\quad\text{or}\quad\quad \psi(z)=c(1-e^{-kz}),\quad\text{ for }c,k>0.
		\end{equation*}
		
		\begin{figure}[ht]
			\includegraphics[scale=0.9]{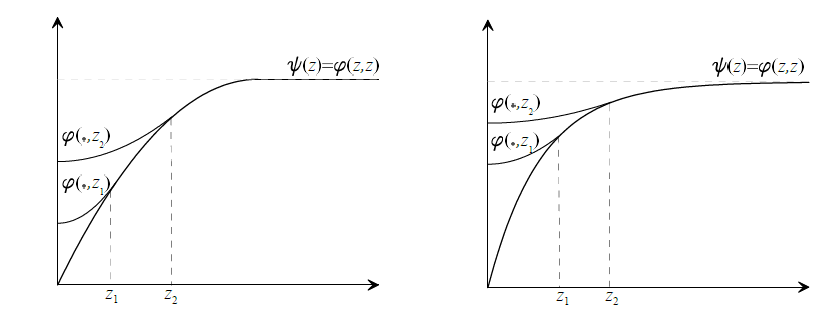}
			\caption{The loading-unloading density $\varphi$ of example \eqref{examplephi} in the cases $\bar{\delta}<+\infty$ (left) and $\bar{\delta}=+\infty$ (right).}\label{Figure}
		\end{figure}
		
		In the first example $\bar\delta=k<+\infty$, while in the second one $\bar\delta=+\infty$.\par
		The function $\varphi$ is then defined by considering a quadratic unloading regime:
		\begin{equation}\label{examplephi}
			\varphi(y,z):=\begin{cases}
				\frac 12 \frac{\psi'(z)}{z}y^2+\psi(z)-\frac 12 z\psi'(z),&\text{ if }(y,z)\in \mc T\setminus(0,0),\\
				0,&\text{ if }(y,z)=(0,0).
			\end{cases}
		\end{equation}
		We refer to Figure~\ref{Figure} for the graphs of $\varphi$. By construction $\varphi$ is continuous on $\mc T$ and ($\varphi$6), ($\varphi$7) and ($\varphi$8) are satisfied. To verify also ($\varphi$9) we notice that it holds:
		\begin{equation*}
			\partial_z\varphi(y,z)=\frac{\psi'(z)-z\psi''(z)}{2}\left(1-\frac{y^2}{z^2}\right),\quad\text{ for every }(y,z)\in \mc T_{\bar\delta}.
		\end{equation*}
		Thus we deduce $\partial_z\varphi$ is continuous in $\mc T_{\bar\delta}\setminus(0,0)$; if moreover $z>y$, since $\psi'(z)$ is strictly positive in $[0,\bar\delta)$, we get
		that $\partial_z\varphi(y,z)>0$ and so ($\varphi$9) is fulfilled.\par
		We finally observe that by the boundedness of $\psi$ we also obtain ($\varphi$2).
	\end{rmk}
	We now present a very simple lemma regarding the behaviour of $\varphi$ in the case $\bar\delta<+\infty$.
	\begin{lemma}
		Assume $\varphi$ satisfies ($\varphi$5), ($\varphi$6) and ($\varphi$7) and assume $\bar\delta$ is finite. Then $\varphi$ is constant in $\mc T\setminus \mc T_{\bar\delta}$,
		and in particular:
		\begin{equation*}
			\varphi(y,z)=\psi(\bar\delta),\text{ for every }(y,z)\in \mc T\setminus\mc T_{\bar\delta}.
		\end{equation*}
	\end{lemma}
	\begin{proof}
		Since $\psi$ is $C^1[0,+\infty)$, then by definition of $\bar\delta$ it holds $\psi'(z)=0$ for every $z\in [\bar\delta,+\infty)$. We now fix $z\in [\bar\delta,+\infty)$;
		by ($\varphi$7) we deduce that $\partial_y\varphi(z,z)=0$. Condition ($\varphi$6) thus yields $\partial_y\varphi(y,z)=0$ for every $y\in [0,z]$, and hence we conclude.
	\end{proof}
	We finally introduce the function $\varphi_{\bar\delta}$, defined as:
	\begin{equation*}
		\varphi_{\bar\delta}(y,z):=\varphi(y\wedge\bar\delta,z\wedge\bar\delta),\quad \text{ for every }(y,z)\in\mc T.
	\end{equation*}
	Thanks to previous lemma, it is easy to deduce that if conditions ($\varphi$5), ($\varphi$6) and ($\varphi$7) are fulfilled, then actually $\varphi$ and $\varphi_{\bar\delta}$
	coincide, namely it holds:
	\begin{equation}\label{equalitiphidelta}
		\varphi(y,z)=\varphi_{\bar\delta}(y,z),\quad \text{ for every }(y,z)\in\mc T.
	\end{equation}
	This last equality will be widely exploited in Section~\ref{sectemporalregularity}.
	\subsection{Energetic evolutions}
	We are now in a position to introduce the notion of solution we want to investigate in this work. Before presenting it we need to consider the prescribed displacement 
	acting on the boundary $x=L$ of the laminate, namely a function $\bar{u}\in AC([0,T])$; we also need to consider initial data for the displacements and damage variables, namely functions $u_i^0,\alpha_i^0$ 
	which must satisfy, for $i=1,2$, the following regularity and compatibility conditions:
	\begin{subequations}\label{compatibility}
		\begin{equation}
			u_i^0,\alpha_i^0\in H^1(0,L),
		\end{equation}
		\begin{equation}
			u_1^0(0)=u_2^0(0)=0,\quad u_1^0(L)=u_2^0(L)=\bar{u}(0),
		\end{equation}
		\begin{equation}
			0\le \alpha_i^0(x)\le 1,\quad \text{ for every }x\in [0,L].
		\end{equation}
	\end{subequations}
	Once the initial displacements are given, we define the initial slip $$\delta^0:=|u_1^0-u_2^0|.$$\par 	
	For $t\in [0,T]$, we denote by $H^1_{0,\bar{u}(t)}(0,L)$ the set of functions $v\in H^1(0,L)$ attaining the boundary values $v(0)=0$ and $v(L)=\bar{u}(t)$.
	We instead denote by $H^1_{[0,1]}(0,L)$ the set of functions $v\in H^1(0,L)$ such that $0\le v(x)\le 1$ for every $x\in[0,L]$.
	\begin{defi}\label{Enersol}
		Given a prescribed displacement $\bar{u}\in AC([0,T])$ and initial data $\bm{u}^0$, $\bm{\alpha}^0$ satisfying \eqref{compatibility}, we say that a bounded pair
		$(\bm{u},\bm{\alpha})\colon [0,T]\times [0,L]\to \erre^2\times\erre^2$ is an \textbf{energetic evolution} if:
		\begin{itemize}
			\item[(CO)] $\bm{u}(t)\in [H^1_{0,\bar{u}(t)}(0,L)]^2$, $\bm{\alpha}(t)\in [H^1_{[0,1]}(0,L)]^2$, for every $t\in[0,T]$;
			\item[(ID)] $\bm{u}(0)=\bm{u}^0$, $\bm{\alpha}(0)=\bm{\alpha}^0$;
			\item[(IR)] for $i=1,2$ the damage function $\alpha_i$ is non-decreasing in time, namely,
			\begin{equation*}
				\text{ for every }0\le s\le t\le T \text{ it holds: }\alpha_i(s,x)\le\alpha_i(t,x), \, \text{ for every }x\in[0,L];
			\end{equation*}
			\item[(GS)] for every $t\in[0,T]$, for every ${\widetilde{\bm u}}\in [H^1_{0,\bar{u}(t)}(0,L)]^2$ and for every ${\widetilde{\bm\alpha}}\in [H^1(0,L)]^2$ such that
			${\alpha}_i(t)\le{\widetilde{\alpha}_i}\le {1}$ in $[0,L]$, $i=1,2$, one has:
			$$\mathcal{E}[\bm{u}(t),\bm{\alpha}(t)]+\mathcal D[\bm{\alpha}(t)]+\mathcal{K}[\delta(t),\delta_h(t)]
			\le\mathcal{E}[{\widetilde{\bm u}},{\widetilde{\bm\alpha}}]+\mathcal D[{\widetilde{\bm\alpha}}]+\mathcal{K}[\widetilde\delta,\delta_h(t)\vee\widetilde\delta];
			$$
			here we mean $\widetilde{\delta}=|\widetilde{u}_1-\widetilde{u}_2|$;
			\item[(EB)] the function $\displaystyle\tau\mapsto \frac{\dot{\bar{u}}(\tau)}{L}\int_{0}^{L}\sum_{i=1}^{2} \sigma_i(\tau,x)\d x$ belongs to $L^1(0,T)$ and for every $t\in[0,T]$ it holds:
			$$\mathcal{E}[\bm{u}(t),\bm{\alpha}(t)]+\mathcal D[\bm{\alpha}(t)]+\mathcal{K}[\delta(t),\delta_h(t)]
			=\mathcal{E}[\bm{u}^0,\bm{\alpha}^0]+\mathcal D[\bm{\alpha}^0]+\mathcal{K}[\delta^0,\delta^0]+\mathcal{W}[\bm{u},\bm{\alpha}](t),$$
			\noindent	
			where 
			\begin{equation}\label{work}
				\mathcal{W}[\bm{u},\bm{\alpha}](t):=\int_{0}^{t}\frac{\dot{\bar{u}}(\tau)}{L}\int_{0}^{L}\sum_{i=1}^{2} \sigma_i(\tau,x)\d x\,d\tau,
			\end{equation}
			is the work done by the external prescribed displacement.
		\end{itemize}
	\end{defi}
	
	In the above Definition (CO) stands for compatibility, (ID) for initial data and (IR) for irreversibility (of the damage variables); the main conditions which characterise this
	sort of solution are of course the global stability (GS) and the energy balance (EB).\par
	We notice that, by (GS), a necessary condition for the existence of such an evolution is the global minimality of the initial data at time $t=0$, namely:
	\begin{equation}\label{GS0}
		\mathcal{E}[\bm{u}^0,\bm{\alpha}^0]+\mathcal D[\bm{\alpha}^0]+\mathcal{K}[\delta^0,\delta^0]
		\le\mathcal{E}[{\widetilde{\bm u}},{\widetilde{\bm\alpha}}]+\mathcal D[{\widetilde{\bm\alpha}}]+\mathcal{K}[\widetilde\delta,\delta^0\vee\widetilde\delta],
	\end{equation}
	for every ${\widetilde{\bm u}}\in [H^1_{0,\bar{u}(0)}(0,L)]^2$ and for every ${\widetilde{\bm\alpha}}\in [H^1(0,L)]^2$ such that ${\alpha}_i^0\le{\widetilde{\alpha}_i}\le {1}$
	in $[0,L]$, $i=1,2$.\par
	We also observe that the definition yields some very weak time regularity on the solution, namely $\bm{u}$ and $\bm{\alpha}$ are bounded in time with values in $[H^1(0,L)]^2$,
	as stated in the next proposition. As a byproduct we also obtain both temporal and spatial regularity on the history variable $\delta_h$, which actually is bounded in time with
	values in  $C^{1/2}_0([0,L])$, namely the space of H\"older-continuous functions with exponent $1/2$ vanishing at $x=0$ and $x=L$. 
	\begin{prop}\label{unifboundsol}
		Assume $E_i$ satisfies \eqref{Ecoerc}, $w_i$ satisfies \eqref{wi}, $\varphi$ satisfies ($\varphi$2) and let $(\bm{u},\bm{\alpha})$ be an energetic evolution.
		Then there exists a positive constant $C$ such that:
		\begin{subequations}
			\begin{equation}\label{bounds1sol}
				\sup\limits_{t\in[0,T]}\Vert\bm{u}(t)\Vert_{[H^1(0,L)]^2}\le \frac{C}{\sqrt{\eps}},\quad\text{ and }\quad\sup\limits_{t\in[0,T]}\Vert\bm{\alpha}(t)\Vert_{[H^1(0,L)]^2}\le C,
			\end{equation}
			where $\eps>0$ has been introduced in \eqref{epsilon}. In particular, $\delta_h(t)$ belongs to $C_0^{1/2}([0,L])$ for every $t\in[0,T]$, and the following estimate holds true:
			\begin{equation}\label{esthold}
				\sup\limits_{t\in[0,T]}|\delta_h(t,x)-\delta_h(t,y)|\le \frac{C}{\sqrt{\eps}}\sqrt{|x-y|},\quad\text{ for every }x,y\in[0,L].
			\end{equation}
		\end{subequations}
	\end{prop}
	\begin{proof}
		Choosing as competitors in (GS) the functions
		\begin{equation*}
			\widetilde{u}_i(x)=\frac{\bar{u}(t)}{L}x,\quad\quad\widetilde\alpha_i\equiv 1,\quad\quad\text{ for }i=1,2,
		\end{equation*}
		and exploiting \eqref{Ecoerc}, \eqref{wi} and ($\varphi$2)  we deduce that
		\begin{align*}
			\frac 12\sum_{i=1}^{2}\left(\eps\Vert u_i(t)'\Vert_{L^2(0,L)}^2+\Vert \alpha_i(t)'\Vert_{L^2(0,L)}^2\right)&\le \mathcal{E}[{\bm{u}(t)},{\bm{\alpha}(t)}]+\mathcal{D}[{\bm{\alpha}(t)}]
			+\mathcal{K}[{\delta}(t),\delta_h(t)]\\
			&\le\mathcal{E}[(\widetilde{u}_1,\widetilde{u}_2), (1,1)]+\mathcal{D}[(1,1)]+\mathcal{K}[0,\delta_h(t)]\\
			&\le C_1\left(\Vert \bar{u}\Vert^2_{C^0([0,T])}+1\right),
		\end{align*}
		for every $t\in[0,T]$, 	where $C_1$ is a suitable positive constant independent of $t$. Since $u_i(t,0)=0$ and $0\le \alpha_i(t,x)\le 1$, we deduce \eqref{bounds1sol}.\par
		By \eqref{bounds1sol} and Sobolev embedding Theorems, we now know that $u_i(t)$ are uniformly H\"older-continuous with exponent $1/2$, for every $t\in[0,T]$.
		We thus fix $t\in[0,T]$ and $x,y\in[0,L]$; by definition of $\delta_h(t,x)$, for every $\eta>0$ there exists $\tau_\eta\in[0,t]$ such that
		\begin{equation*}
			\delta_h(t,x)-\eta\le |u_1(\tau_\eta,x)-u_2(\tau_\eta,x)|.
		\end{equation*}
		Hence we can estimate:
		\begin{align*}
			\delta_h(t,x)-\eta\le& |u_1(\tau_\eta,x)-u_1(\tau_\eta,y)|+|u_1(\tau_\eta,y)-u_2(\tau_\eta,y)|+|u_2(\tau_\eta,y)-u_2(\tau_\eta,x)|\\
			\le& \frac{C}{\sqrt{\eps}}\sqrt{|x-y|}+\delta_h(t,y),
		\end{align*}
		for any $t\in[0,T]$ and $x,y\in[0,L]$. By the arbitrariness of $\eta$ and reverting the role of $x$ and $y$ we deduce that $\delta_h(t)$ is H\"older-continuous with exponent $1/2$ and \eqref{esthold} holds true.
		Trivially $\delta_h(t,0)=\delta_h(t,L)=0$ and so we conclude.
	\end{proof}
	\begin{rmk}
		In the previous proposition we stressed the dependence on $\eps>0$  to point out the importance of assumption \eqref{Ecoerc}, which ensures the coerciveness of the elastic energy.
		In the complete damage setting, where $E_i$ can vanish, one needs to consider the sequence of functions $E_i+\eps$, fulfilling \eqref{Ecoerc}, and then to perform an analysis of the limit $\eps \to 0^+$, usually via $\Gamma$--convergence \cite{DalMasoGamma}. We refer for instance to \cite{BouchMielkRoub,MielkRoub} for a model of contact between two viscoelastic bodies, or to \cite{BonFreSeg}.
	\end{rmk}
	As we said in the Introduction, the common procedure used to prove existence of energetic evolutions (and which we will perform in Section~\ref{secexistence}) is based on a time
	discretisation algorithm and then on a limit passage as the time step goes to $0$. Due to lack of compactness for the history variable $\delta_h$, one needs to weaken the notion
	of energetic evolution and to introduce a fictitious variable $\gamma$ replacing $\delta_h$ (see also \cite{CagnToad, DMZan}). Thanks to Proposition~\ref{unifboundsol} we however
	expect that $\gamma(t)$ should be at least continuous in $[0,L]$; we are thus led to the following definition:
	\begin{defi}\label{Generalisedenersol}
		Given a prescribed displacement $\bar{u}\in AC([0,T])$ and initial data $\bm{u}^0$, $\bm{\alpha}^0$ satisfying \eqref{compatibility}, we say that a triple
		$(\bm{u},\bm{\alpha},\gamma)\colon[0,T]\times[0,L]\to \erre^2\times\erre^2\times\erre$ is a \textbf{generalised energetic evolution} if:
		\begin{itemize}
			\item[(CO')] $\bm{u}(t)\in [H^1_{0,\bar{u}(t)}(0,L)]^2$, $\bm{\alpha}(t)\in [H^1_{[0,1]}(0,L)]^2$, $\gamma(t)\in C^0([0,L])$, for every $t\in[0,T]$;
			\item[(ID')] $\bm{u}(0)=\bm{u}^0$, $\bm{\alpha}(0)=\bm{\alpha}^0$, $\gamma(0)=\delta^0$;
			\item[(IR')] for $i=1,2$ the damage function $\alpha_i$ and the generalised history variable $\gamma$ are non-decreasing in time, namely,
			\begin{equation*}
				\text{ for every }0\le s\le t\le T	\text{ it holds: } \alpha_i(s,x)\le\alpha_i(t,x), \,\text{ for every }x\in[0,L];
			\end{equation*}
			\begin{equation*}
				\text{ for every }0\le s\le t\le T	\text{ it holds: } \gamma(s,x)\le\gamma(t,x), \,	\text{ for every }x\in[0,L];
			\end{equation*}
			\item[(GS')] for every $t\in[0,T]$ one has $\gamma(t)\ge \delta(t)$ in $[0,L]$ and:
			\begin{equation*}
				\mathcal{E}[\bm{u}(t),\bm{\alpha}(t)]+\mathcal D[\bm{\alpha}(t)]+\mathcal{K}[\delta(t),\gamma(t)]\le\mathcal{E}[{\widetilde{\bm u}},{\widetilde{\bm\alpha}}]
				+\mathcal D[{\widetilde{\bm\alpha}}]+\mathcal{K}[\widetilde\delta,\gamma(t)\vee\widetilde\delta],
			\end{equation*}
			for every ${\widetilde{\bm u}}\in [H^1_{0,\bar{u}(t)}(0,L)]^2$ and for every ${\widetilde{\bm\alpha}}\in [H^1(0,L)]^2$ such that ${\alpha}_i(t)\le{\widetilde{\alpha}_i}
			\le {1}$ in $[0,L]$ for $i=1,2$;
			\item[(EB')] the function $\displaystyle\tau\mapsto \frac{\dot{\bar{u}}(\tau)}{L}\int_{0}^{L}\sum_{i=1}^{2} \sigma_i(\tau,x)\d x$ belongs to $L^1(0,T)$ and for every $t\in[0,T]$ it holds:
			\begin{equation*}
				\mathcal{E}[\bm{u}(t),\bm{\alpha}(t)]+\mathcal D[\bm{\alpha}(t)]+\mathcal{K}[\delta(t),\gamma(t)]=\mathcal{E}[\bm{u}^0,\bm{\alpha}^0]+\mathcal D[\bm{\alpha}^0]
				+\mathcal{K}[\delta^0,\delta^0]+\mathcal{W}[\bm{u},\bm{\alpha}](t),
			\end{equation*}
			where $\mathcal{W}[\bm{u},\bm{\alpha}](t)$ is defined as in  \eqref{work}.
		\end{itemize}
	\end{defi}
	\begin{rmk}\label{Rmkimportant}
		If conditions ($\varphi$5), ($\varphi$6) and ($\varphi$7) are satisfied, then equality \eqref{equalitiphidelta} allows us to replace the function $\varphi$ in the functional
		$\mc K$ (see \eqref{Cohesiveenergy}) by $\varphi_{\bar\delta}$. This means that the functions which actually play a role in the cohesive energy are
		$\delta\wedge\bar\delta$, $\delta_h\wedge\bar\delta$ and $\gamma\wedge\bar\delta$. This observation will be useful in Section~\ref{sectemporalregularity}.
	\end{rmk}
	From the very definition it is easy to see that a pair $(\bm{u},\bm{\alpha})$ is an energetic evolution if and only if the triple $(\bm{u},\bm{\alpha},\delta_h)$ is a generalised
	energetic evolution. It is also easy to see that given a generalised energetic evolution $(\bm{u},\bm{\alpha},\gamma)$ it necessarily holds $\gamma(t,x)\ge \delta_h(t,x)$,
	for every $(t,x)\in[0,T]\times[0,L]$. Unfortunately, there are no easy arguments which ensure that $\gamma=\delta_h$ in a general case. This will be the topic of
	Section~\ref{sectemporalregularity} and the main outcome of the paper.\par
	We finally notice that the same argument used to prove Proposition~\ref{unifboundsol} leads to the bound \eqref{bounds1sol} also for a generalised energetic evolution.
	However \eqref{esthold} only holds for $\delta_h$ due to its explicit definition \eqref{historicalslip}, and nothing can be said, in general, about the generalised history
	variable $\gamma$.
	
	\section{Existence Result}\label{secexistence}
	In this section we show existence of generalised energetic evolutions under very weak assumptions on the data, especially on the density $\varphi$.
	We indeed require \eqref{Ecoerc}, \eqref{wi} and only ($\varphi$1), ($\varphi$2), ($\varphi$3), see Theorem~\ref{exgenensol}. Of course we always assume that the prescribed
	displacement $\bar u$ belongs to $AC([0,T])$. We then prove the existence of an energetic evolution assuming the specific assumption ($\varphi$4), following the same approach
	of \cite{CagnToad}, see Theorem~\ref{exensol}. We will overcome the necessity of ($\varphi$4) in Section~\ref{sectemporalregularity}, recovering the existence of energetic evolutions in meaningful mechanical
	situations (namely assuming ($\varphi$5)--($\varphi$9), see also Remark~\ref{rmkexamplephi}) and thus obtaining our main result, Theorem~\ref{finalthm}.\par
	The classical tool used to prove existence of energetic evolutions is a time--discretisation procedure. Here we combine the ideas of \cite{MielkRoub} to deal with the irreversible
	damage variables and of \cite{CagnToad, DMZan} to handle the history variable.
	\subsection{Time-discretisation}\label{subseq1}
	We consider a sequence of partition $0=t^n_0<t^n_1<\dots<t^n_n=T$ such that
	\begin{equation}\label{finepartition}
		\lim\limits_{n\to+\infty}\max\limits_{k=1,\cdots,n}(t^n_{k}-t^n_{k-1})=0,
	\end{equation}
	and for $k=1,\cdots,n$ we perform the following implicit Euler scheme: given $(\bm{u}^{k-1},\bm{\alpha}^{k-1}, \delta_h^{k-1})$, we first select
	$(\bm{u}^k,\bm{\alpha}^k)$ by minimising the total energy among suitable natural competitors:
	\begin{subequations}
		\begin{equation}\label{minalg}
			(\bm{u}^k,\bm{\alpha}^k)\in\argmin\limits_{\substack{\widetilde{\bm{u}}\in[H^1_{0,\bar{u}(t^n_k)}(0,L)]^2,\\\widetilde{\bm{\alpha}}\in [H^1(0,L)]^2\,
					\text{s.t. }{\alpha}_i^{k-1}\le\widetilde{{\alpha}}_i\le 1}}\Big\{\mathcal{E}[\widetilde{\bm{u}},\widetilde{\bm{\alpha}}]
			+\mathcal{D}[\widetilde{\bm{\alpha}}]+\mathcal{K}[\widetilde{\delta},\delta_h^{k-1}\vee\widetilde{\delta}]\Big\}.
		\end{equation}
		Here we want to recall that we mean $\widetilde{\delta}=|\widetilde{u}_1-\widetilde{u}_2|$.\par
		We then define $\delta_h^k$ as:
		\begin{equation}
			\delta_h^k:=\delta_h^{k-1}\vee|u_1^k-u_2^k|=\delta_h^{k-1}\vee\delta^k.
		\end{equation}
	\end{subequations}
	The initial values in the minimisation algorithm are functions $(\bm{u}^0,\bm{\alpha}^0)$ satisfying the compatibility conditions \eqref{compatibility};
	moreover we set $\delta_h^0:=\delta^0=|u_1^0-u_2^0|$.
	
	\begin{prop}
		Assume $E_i$ satisfies \eqref{Ecoerc}, $w_i$ satisfies \eqref{wi} and $\varphi$ satisfies ($\varphi$1). Then there exists a solution to the minimisation algorithm \eqref{minalg}.
	\end{prop}
	\begin{proof}
		We fix $n\in\enne$ and for every $k=1,\dots,n$ we prove the existence of a minimum by means of the direct method of Calculus of Variations.
		For the sake of clarity we denote by $\mc F^{k-1}$ the functional we want to minimise, namely
		\begin{equation}
			\mc F^{k-1}[\widetilde{\bm u},\widetilde{\bm{\alpha}}]=
			\mathcal{E}[\widetilde{\bm{u}},\widetilde{\bm{\alpha}}]+\mathcal{D}[\widetilde{\bm{\alpha}}]+\mathcal{K}[\widetilde{\delta},\delta_h^{k-1}\vee\widetilde{\delta}]+\chi_{A^{k-1}}[\widetilde{\bm u},\widetilde{\bm{\alpha}}],
		\end{equation}
		where $\chi_{A^{k-1}}$ denotes the indicator function of the set of constraints $A^{k-1}$, which is given by
		\begin{equation*}
			A^{k-1}:=\{(\widetilde{\bm u},\widetilde{\bm{\alpha}})\in [H^1_{0,\bar{u}(t^n_k)}(0,L)]^2\times[H^1(0,L)]^2\mid {\alpha}_i^{k-1}(x)\le\widetilde{{\alpha}}_i(x)\le 1 \text{ for every }x\in[0,L]\}.
		\end{equation*}
		Weak (sequential) compactness in $[H^1(0,L)]^4$ for a minimising sequence for $\mc F^{k-1}$ follows by means of uniform bounds which can be obtained by reasoning as in the proof of Proposition~\ref{unifboundsol}.
		
		As regards the (sequential) lower semicontinuity of $\mathcal{F}^{k-1}$ with respect the considered topology we exploit the compact embedding $H^1(0,L)\subset\subset C^0({0,L})$.
		By ($\varphi$1) and Fatou's Lemma we thus deduce that $\mc K$ is lower semicontinuous; the same holds true for
		$\mc D$ by using again Fatou's Lemma together with weak lower semicontinuity of the norm. To prove lower semicontinuity of $\mc E$ it is enough to show that,
		given weakly convergent sequences $\widetilde{u}_i^j\rightharpoonup\widetilde{u}_i$, $\widetilde{\alpha}_i^j\rightharpoonup\widetilde{\alpha}_i$ in $H^1(0,L)$,
		we have that $\sqrt{E_i(\widetilde{\alpha}_i^j)}(\widetilde{u}_i^j){'}$ weakly converges to $\sqrt{E_i(\widetilde{\alpha}_i)}{\widetilde{u}_i}{'}$ in $L^2(0,L)$ as $j\to +\infty$,
		for $i=1,2$. To prove it we fix $\phi\in L^2(0,L)$ and we estimate by exploiting \eqref{Ecoerc}:
		\begin{align*}
			&\quad\left|\int_{0}^{L}\sqrt{E_i(\widetilde{\alpha}_i^j(x))}(\widetilde{u}_i^j){'}(x)\phi(x)\d x
			-\int_{0}^{L}\sqrt{E_i(\widetilde{\alpha}_i(x))}{\widetilde{u}_i}{'}(x)\phi(x)\d x\right|\\
			&\le\left\Vert\sqrt{E_i(\widetilde{\alpha}_i^j)}-\sqrt{E_i(\widetilde{\alpha}_i)}\right\Vert_{C^0([0,L])}\Vert\widetilde{u}_i^j\Vert_{H^1(0,L)}\Vert\phi\Vert_{L^2(0,L)}\\
			&\quad+\left|\int_{0}^{L}(\widetilde{u}_i^j){'}(x)\sqrt{E_i(\widetilde{\alpha}_i(x))}\phi(x)\d x
			-\int_{0}^{L}{\widetilde{u}_i}{'}(x)\sqrt{E_i(\widetilde{\alpha}_i(x))}\phi(x)\d x\right|.
		\end{align*}
		The first term goes to zero as $j\to +\infty$ since $\widetilde{\alpha}_i^j$ uniformly converges to $\widetilde{\alpha}_i$ as $j\to +\infty$ and the function $E_i$ is continuous.
		The second term vanishes too as $j\to +\infty$ since $\sqrt{E_i(\widetilde{\alpha}_i)}\phi$ belongs to $L^2(0,L)$ by the boundedness of $E_i$.\par
		We conclude by noticing that, exploiting again the compactness of the embedding $H^1(0,L)\subset\subset C^0({0,L})$, the set $A^{k-1}$ is (sequentially) closed with respect to the considered topology, and thus its indicator function $\chi_{A^{k-1}}$ is lower semicontinuous as well.
	\end{proof}\noindent
	To pass from discrete to continuous evolutions we now introduce the (right-continuous) piecewise constant interpolants $(\bm{u}^n,\bm{\alpha}^n)$ of the discrete displacement and
	damage variables, and the piecewise constant interpolant $\delta^n_h$ of the discrete history variable, namely:
	\begin{subequations}\label{interpolation}
		\begin{equation}\label{interpolationa}
			\begin{cases}
				\bm{u}^n(t):=\bm{u}^k,\quad\,\,\bm{\alpha}^n(t):=\bm{\alpha}^k,\quad \,\,\delta_h^n(t):=\delta_h^k,\quad\quad\text{if }t\in[t^n_k,t^n_{k+1}),\\
				\bm{u}^n(T):=\bm{u}^n,\quad\bm{\alpha}^n(T):=\bm{\alpha}^n,\quad\delta_h^n(T):=\delta_h^n.
			\end{cases}
		\end{equation}
		Of course, in the following, by the expression $\delta^n$ we mean the piecewise constant slip, namely \begin{equation}
			\delta^n(t,x)=|u_1^n(t,x)-u_2^n(t,x)|.
		\end{equation}	
		Analogously, we consider a piecewise constant version $\bar{u}^n$ of the prescribed displacement:
		\begin{equation}
			\begin{cases}
				\bar{u}^n(t):=\bar{u}(t^n_k),\quad\text{if }t\in[t^n_k,t^n_{k+1}),\\
				\bar{u}^n(T):=\bar{u}(T).
			\end{cases}
		\end{equation}
		We also adopt the following notation:
		\begin{equation}
			\tau^n(t):=\max\{t^n_k\mid t^n_k\le t\}.
		\end{equation}
	\end{subequations}
	The next proposition provides useful uniform bounds on the just introduced piecewise constant interpolants. It is the analogue of Proposition~\ref{unifboundsol} in this discrete setting.
	\begin{prop}\label{unifbound}
		Assume $E_i$ satisfies \eqref{Ecoerc}, $w_i$ satisfies \eqref{wi} and $\varphi$ satisfies ($\varphi$1), ($\varphi$2). Then there exists a positive constant $C$ independent of $n$
		such that:
		\begin{subequations}
			\begin{equation}\label{bounds1}
				\max\limits_{t\in[0,T]}\Vert\bm{u}^n(t)\Vert_{[H^1(0,L)]^2}\le \frac{C}{\sqrt{\eps}},\quad \quad\max\limits_{t\in[0,T]}\Vert\bm{\alpha}^n(t)\Vert_{[H^1(0,L)]^2}\le C,
			\end{equation}
			\begin{equation}\label{bounds2}
				\max\limits_{t\in[0,T]}\left(\sup\limits_{x,y\in[0,L],\, x\neq y}\frac{|\delta_h^n(t,x)-\delta_h^n(t,y)|}{\sqrt{|x-y|}}\right)\le \frac{C}{\sqrt{\eps}},
			\end{equation}
		\end{subequations}
		where $\eps>0$ has been introduced in \eqref{epsilon}.
	\end{prop}
	\begin{proof}
		The result follows by using exactly the same argument of Proposition~\ref{unifboundsol}. We only notice that here we need to choose as competitors for $(\bm{u}^k,\bm{\alpha}^k)$ 
		in \eqref{minalg} the functions
		\begin{equation*}
			\widetilde{u}_i(x)=\frac{\bar{u}(t^n_k)}{L}x,\quad\quad\widetilde\alpha_i\equiv 1,\quad\quad\text{ for }i=1,2,
		\end{equation*}
		and then we argue in the same way.
	\end{proof}
	Since the piecewise constant interpolants are built starting from the minimisation algorithm \eqref{minalg}, they automatically fulfil the following inequality, which is related to 
	the energy balance (EB):
	\begin{lemma}[\textbf{Discrete Energy Inequality}]\label{discreteenergyineq}
		Assume $E_i$ satisfies \eqref{Ecoerc}, $w_i$ satisfies \eqref{wi} and $\varphi$ satisfies ($\varphi$1), ($\varphi$2). Then there exists a vanishing sequence of positive 
		real numbers $R^n$ such that for every $t\in[0,T]$ and for every $n\in\enne$ the following inequality holds true:
		\begin{align*}
			&\quad\,\,\mathcal{E}[{\bm{u}^n(t)},{\bm{\alpha}^n(t)}]+\mathcal{D}[{\bm{\alpha}^n(t)}]+\mathcal{K}[{\delta}^n(t),\delta_h^{n}(t)]\\
			&\le \mathcal{E}[{\bm{u}^0},{\bm{\alpha}^0}]+\mathcal{D}[{\bm{\alpha}^0}]+\mathcal{K}[{\delta}^0,\delta^0]+\int_{0}^{t}W^{n}(\tau)d\tau+R^n,
		\end{align*}
		where $\displaystyle W^{n}(\tau):=\frac{\dot{\bar{u}}(\tau)}{L}\sum_{i=1}^{2}\int_{0}^{L}E_i(\alpha^n_i(\tau,x))(u_i^n)'(\tau,x)\d x$.
	\end{lemma}
	\begin{proof}
		We fix $n\in \enne$ and $k\in\{1,\cdots,n\}$; for $j=1,\cdots,k$ we then choose as competitors for $(\bm{u}^j,\bm{\alpha}^j)$ in \eqref{minalg} the functions
		$\widetilde{\bm{u}}$, $\widetilde{\bm{\alpha}}$, with components:
		\begin{equation*}
			\widetilde{u}_i(x)={u}_i^{j-1}(x)+(\bar{u}(t^n_j)-\bar{u}(t^n_{j-1}))x/L,\quad\text{ and }\widetilde{{\alpha}}_i={\alpha}_i^{j-1},\text{ for }i=1,2.
		\end{equation*}
		We thus obtain:
		\begin{align*}
			\mathcal{E}[{\bm{u}^j},{\bm{\alpha}^j}]+\mathcal{D}[{\bm{\alpha}^j}]+\mathcal{K}[{\delta}^j,\delta_h^{j}]\le\mathcal{E}[\bm{u}^{j-1}+\bm{v}^{j-1},\bm{\alpha}^{j-1}]
			+\mathcal{D}[\bm{\alpha}^{j-1}]+\mathcal{K}[\delta^{j-1},\delta_h^{j-1}],
		\end{align*}
		where we denoted by $\bm{v}^{j-1}(x)$ the vector in $\erre^2$ with both components equal to $(\bar{u}(t^n_j)-\bar{u}(t^n_{j-1}))\frac xL$. From the above inequality we now get:
		\begin{align*}
			&\mathcal{E}[{\bm{u}^j},{\bm{\alpha}^j}]+\mathcal{D}[{\bm{\alpha}^j}]
			+\mathcal{K}[{\delta}^j,\delta_h^{j}]-\mathcal{E}[{\bm{u}^{j-1}},{\bm{\alpha}^{j-1}}]-\mathcal{D}[{\bm{\alpha}^{j-1}}]-\mathcal{K}[{\delta}^{j-1},\delta_h^{{j-1}}]\\
			\le&\mathcal{E}[\bm{u}^{j-1}+\bm{v}^{j-1},\bm{\alpha}^{j-1}]-\mathcal{E}[\bm{u}^{j-1},\bm{\alpha}^{j-1}]\\
			=&\int_{t^n_{j-1}}^{t^n_j}\frac{\dot{\bar{u}}(\tau)}{L}\sum_{i=1}^{2}\int_{0}^{L}E_i(\alpha_i^{j-1}(x))\left((u_i^{j-1})'(x)
			+\frac{\bar{u}(\tau)-\bar{u}(t^n_{j-1})}{L}\right)\d x\d\tau.
		\end{align*}
		Summing the obtained inequality from $j=1$ to $j=k$ we hence deduce:
		\begin{align*}
			&\mathcal{E}[{\bm{u}^k},{\bm{\alpha}^k}]+\mathcal{D}[{\bm{\alpha}^k}]+\mathcal{K}[{\delta}^k,\delta_h^{k}]
			-\mathcal{E}[{\bm{u}^{0}},{\bm{\alpha}^{0}}]-\mathcal{D}[{\bm{\alpha}^{0}}]-\mathcal{K}[{\delta}^{0},\delta^{{0}}]\\
			\le&\sum_{j=1}^{k}\left(\int_{t^n_{j-1}}^{t^n_j}W^{n}(\tau)\d\tau
			+\int_{t^n_{j-1}}^{t^n_j}\frac{\dot{\bar{u}}(\tau)}{L}\frac{\bar{u}(\tau)-\bar{u}^n(\tau)}{L}\sum_{i=1}^{2}\int_{0}^{L}E_i(\alpha_i^{n}(\tau,x))\d x \d\tau\right)\\
			=&\int_{0}^{t^n_k}W^{n}(\tau)\d\tau
			+\int_{0}^{t^n_k}\frac{\dot{\bar{u}}(\tau)}{L}\frac{\bar{u}(\tau)-\bar{u}^n(\tau)}{L}\sum_{i=1}^{2}\int_{0}^{L}E_i(\alpha_i^{n}(\tau,x))\d x \d\tau.
		\end{align*}
		Recalling the definition of the interpolants $\bm{u}^n$, $\bm{\alpha}^n$ and $\tau^n$, see \eqref{interpolation}, by the arbitrariness of $k$ we finally obtain
		for every $t\in[0,T]$:
		\begin{align*}
			&\quad\,\,\mathcal{E}[{\bm{u}^n(t)},{\bm{\alpha}^n(t)}]+\mathcal{D}[{\bm{\alpha}^n(t)}]+\mathcal{K}[{\delta}^n(t),\delta_h^{n}(t)]\\
			&\le \mathcal{E}[{\bm{u}^0},{\bm{\alpha}^0}]+\mathcal{D}[{\bm{\alpha}^0}]+\mathcal{K}[{\delta}^0,\delta^0]+\int_{0}^{t}W^{n}(\tau)d\tau\\
			&\quad+\int_{0}^{\tau^n(t)}\frac{\dot{\bar{u}}(\tau)}{L}\frac{\bar{u}(\tau)-\bar{u}^n(\tau)}{L}\sum_{i=1}^{2}\int_{0}^{L}E_i(\alpha_i^{n}(\tau,x))\d x \d\tau
			-\int_{\tau^n(t)}^{t}W^{n}(\tau)\d\tau.
		\end{align*}
		We thus conclude by defining:
		\begin{equation}\label{Rn}
			R^n:=\int_{0}^{T}\frac{|\dot{\bar{u}}(\tau)|}{L}\frac{|\bar{u}(\tau)-\bar{u}^n(\tau)|}{L}\sum_{i=1}^{2}\int_{0}^{L}E_i(\alpha_i^{n}(\tau,x))\d x\d\tau
			+\sup\limits_{t\in[0,T]}\int_{\tau^n(t)}^{t}|W^{n}(\tau)|\d\tau.
		\end{equation}
		Indeed we now show that $\lim\limits_{n\to +\infty} R^n=0$. First of all by the very definition of $W^{n}$ and exploiting \eqref{bounds1} it is easy to see that
		$|W^{n}(\tau)|\le C|\dot{\bar{u}}(\tau)|$, with $C>0$ independent of $n$; hence by the absolute continuity of the integral the second term in \eqref{Rn} vanishes as
		$n\to +\infty$ (we recall that by assumption the sequence of partitions satisfies \eqref{finepartition}). Then we notice that the first term is bounded by
		\begin{equation*}
			C\Vert \dot{\bar{u}}\Vert_{L^1(0,T)}\sup\limits_{t\in[0,T]}|\bar{u}(t)-\bar{u}^n(t)|,
		\end{equation*}
		which vanishes since $\bar{u}$ is absolutely continuous and the sequence of partitions satisfies \eqref{finepartition}.
	\end{proof}
	\subsection{Extraction of convergent subsequences}
	By the uniform bounds obtained in Proposition~\ref{unifbound} we are able to  deduce the existence of convergent subsequences of the piecewise constant interpolants
	$\bm{u}^{n}$, $\bm{\alpha}^{n}$ and $\delta_h^{n}$. We first need the following Helly--type compactness result:
	\begin{lemma}[\textbf{Helly}]\label{Helly}
		Let $\{f_n\}_{n\in\enne}$ be a sequence of non-decreasing functions from $[0,T]$ to $C^0([0,L])$, meaning that for every $0\le s\le t\le T$ it holds $f_n(s,x)\le f_n(t,x)$ for all $x\in[0,L]$, such that:
		\begin{itemize}
			\item the families $\{f_n(0)\}_{n\in\enne}$ and $\{f_n(T)\}_{n\in\enne}$ are equibounded;
			\item the family $\{f_n(t)\}_{n\in\enne}$ is equicontinuous uniformly with respect to $t\in[0,T]$.
		\end{itemize}
		Then there exist a subsequence (not relabelled) and a function $f\colon [0,T]\to C^0([0,L])$ such that $f_n(t)$ converges uniformly to $f(t)$ as $n\to+\infty$ for every $t\in[0,T]$, and $f$ is non-decreasing in time, in the above sense.\par
		Moreover for every $t\in[0,T]$ the right and left limits $f^\pm(t)$, which are well defined pointwise by monotonicity, actually belong to $C^0([0,L])$ and it holds
		\begin{equation}\label{unifrightleftlimits}
			f^\pm (t)=\lim\limits_{h\to 0^\pm}f(t+h),\quad\text{ uniformly in } [0,L].
		\end{equation}
	\end{lemma}
	\begin{proof}
		The proof follows exactly the same lines of Lemma~4.6 in \cite{DMZan}; we only stress two differences. Here, the topology is the one inherited by uniform convergence and compactness is ensured by the Ascoli--Arzel\'a theorem, thanks to the equiboundedness and equicontinuity assumptions. The additional requirement of uniform equicontinuity with respect to $t\in[0,T]$ is finally used to deduce that the limit family $\{f(t)\}_{t\in [0,T]}$ is equicontinuous as well, thus yielding \eqref{unifrightleftlimits}.
	\end{proof}
	\begin{prop}\label{limits}
		Assume $E_i$ satisfies \eqref{Ecoerc}, $w_i$ satisfies \eqref{wi} and $\varphi$ satisfies ($\varphi$1), ($\varphi$2). Consider the sequences of functions $\bm{u}^n$, $\bm{\alpha}^n$, $\delta^n_h$ introduced in \eqref{interpolationa}. Then there exist a subsequence $n_j$ and for every $t\in[0,T]$ a further subsequence $n_j(t)$ (depending on time) such that:
		\begin{itemize}
			\item[(a)] $\bm{u}^{n_j(t)}(t)\rightharpoonup\bm{u}(t)$ in $[H^1(0,L)]^2$ as ${n_j(t)\to+\infty}$;
			\item[(b)] $\bm{\alpha}^{n_j(t)}(t)\rightharpoonup\bm{\alpha}(t)$ in $[H^1(0,L)]^2$ as ${n_j(t)\to+\infty}$;
			\item[(c)] $\delta_h^{n_j}(t)\to \gamma(t)$ uniformly in $[0,L]$ as ${n_j\to+\infty}$.
		\end{itemize}
		Moreover the limit functions satisfy:
		\begin{enumerate}
			\item $\bm{u}(t)\in [H^1_{0,\bar{u}(t)}(0,L)]^2$, $\bm{\alpha}(t)\in [H^1_{[0,1]}(0,L)]^2$ and $\gamma(t)\in C_{0}^{1/2}([0,L])$ for every $t\in[0,T]$;
			\item $\bm{u}(0)=\bm{u}^0$, $\bm\alpha(0)=\bm\alpha^0$ and $\gamma(0)=\delta^0$;
			\item $\alpha_i$ and $\gamma$ are non-decreasing in time;
			\item $\gamma(t)\ge \delta_h(t)=\sup\limits_{\tau\in[0,t]}|u_1(\tau)-u_2(\tau)|$ for every $t\in[0,T]$;
			\item the family $\{\gamma(t)\}_{t\in[0,T]}$ is equicontinuous.
		\end{enumerate}
	\end{prop}
	\begin{rmk}
		We want to point out that also the subsequence of the damage variable in (b) could be chosen independent of time, since each term of the sequence is non-decreasing in time.
		This follows by means of a suitable version of Helly's selection theorem (see for instance Theorem B.5.13 in the Appendix B of \cite{MielkRoubbook}), and arguing as in
		\cite{MielkRoub}, Proposition~3.2. However, both for the sake of simplicity and since for (a) the same can not be done, we prefer to consider a time--dependent subsequence; this will be enough for our purposes.\par
		The fact that the subsequence in (c) does not depend on time is instead crucial for the validity of (4), as the reader can check from the proof.\par
	\end{rmk}
	\begin{rmk}\label{subsequencermk}
		For the sake of clarity, in order to avoid too heavy notations, from now on we prefer not to stress the occurence of the subsequence via the subscript $j$; namely we still write $n$ instead of $n_j$ and $n(t)$ instead of $n_j(t)$.
	\end{rmk}
	\begin{proof}[Proof of Proposition~\ref{limits}]
		The validity of (c) and (5), the H\"older-continuity of exponent $1/2$ of the limit function $\gamma(t)$ and the fact that $\gamma(t,0)=\gamma(t,L)=0$ are a byproduct of \eqref{bounds2} and
		Lemma~\ref{Helly}; (a) and (b) instead follow by \eqref{bounds1} together with the weak sequential compactness of the unit ball in $H^1(0,L)$.
		Since $H^1(0,L)\subset\subset C^0([0,L])$ we also deduce (1), (2) and (3).\par
		We only need to prove (4). So let us assume by contradiction that there exists a pair $(t,x)\in[0,T]\times[0,L]$ such that:
		\begin{equation}\label{contr}
			\delta_h(t,x)>\gamma(t,x)=\lim\limits_{n\to +\infty}\delta^{n}_h(t,x).
		\end{equation}
		By \eqref{contr} and the definition of $\delta_h$, there exists a time $\tau_t\in[0,t]$ for which $|u_1(\tau_t,x)-u_2(\tau_t,x)|>\gamma(t,x)$; thus we infer:
		\begin{align*}
			|u_1(\tau_t,x)-u_2(\tau_t,x)|&>\lim\limits_{n\to +\infty}\delta^{n}_h(t,x)\ge\lim\limits_{n\to +\infty}\delta^{n}_h(\tau_t,x)\\
			&\ge \limsup\limits_{n\to +\infty}|u_1^{n}(\tau_t,x)-u_2^{n}(\tau_t,x)|\\
			&\ge \lim\limits_{n(\tau_t)\to +\infty}|u_1^{n(\tau_t)}(\tau_t,x)-u_2^{n(\tau_t)}(\tau_t,x)|=|u_1(\tau_t,x)-u_2(\tau_t,x)|,
		\end{align*}
		which is a contradiction.
	\end{proof}
	
	\subsection{Existence of generalised energetic evolutions}
	The aim of this subsection is proving that the limit functions obtained in Proposition~\ref{limits} are actually a generalised energetic evolution. We only need to show that
	global stability (GS') and energy balance (EB') hold true, being the other conditions automatically satisfied due to Lemma~\ref{limits}. This first proposition deals with the
	global stability:
	\begin{prop}\label{globstab}
		Assume $E_i$ satisfies \eqref{Ecoerc}, $w_i$ satisfies \eqref{wi}, and $\varphi$ satisfies ($\varphi$1)--($\varphi$3). Assume the initial data $\bm{u}^0$, $\bm{\alpha}^0$
		fulfil the stability condition \eqref{GS0}. Then the limit functions $\bm{u}$, $\bm{\alpha}$, $\gamma$ obtained in Proposition~\ref{limits} satisfy (GS').
	\end{prop}
	\begin{proof}
		If $t=0$ there is nothing to prove, so we consider $t\in(0,T]$ and we first notice that by (4) in Proposition~\ref{limits} we know $\gamma(t)\ge\delta(t)$.
		Then we fix $\widetilde{\bm{u}}\in [H^1_{0,\bar{u}(t)}(0,L)]^2$ and $\widetilde{\bm{\alpha}}\in [H^1(0,L)]^2$ such that $\alpha_i(t)\le\widetilde{\alpha}_i\le 1$ for $i=1,2$.\par
		By weak lower semicontinuity of the energy, taking the subsequence $n(t)$ obtained in Proposition~\ref{limits} (see also Remark~\ref{subsequencermk}), we get:
		\begin{align*}
			&\quad\mathcal{E}[\bm{u}(t),\bm{\alpha}(t)]+\mathcal D[\bm{\alpha}(t)]+\mathcal{K}[\delta(t),\gamma(t)]\\
			\le& \liminf\limits_{n(t)\to+\infty}\left(\mathcal{E}[\bm{u}^{n(t)}(t),\bm{\alpha}^{n(t)}(t)]+\mathcal D[\bm{\alpha}^{n(t)}(t)]
			+\mathcal{K}[\delta^{n(t)}(t),\delta_h^{n(t)}(t)]\right)=:(\star).
		\end{align*}
		Now we can use the minimality properties of the discrete functions, considering as competitors the functions $\widehat{\bm{u}}^{n(t)}$ and $\widehat{\bm{\alpha}}^{n(t)}$
		whose components are
		\begin{equation*}
			\widehat{{u}}^{n(t)}_i(x):=\widetilde{u}_i(x)-(\bar{u}(t)-\bar{u}(\tau^{n(t)}(t))\frac{x}{L},\quad\quad \widehat{\alpha}_i^{n(t)}:=\min\left\{\widetilde{\alpha}_i
			+\max\limits_{[0,L]}\left|\alpha_i^{n(t)}(t)-\alpha_i(t)\right|,1\right\}.
		\end{equation*}
		It is easy to see that they are admissible; moreover, since $\tau^{n(t)}(t)\to t$ and $\alpha_i^{n(t)}(t)\to\alpha_i(t)$ uniformly as $n(t)\to +\infty$, they strongly
		converge to $\widetilde{\bm{u}}$ and $\widetilde{\bm\alpha}$ in $[H^1(0,L)]^2$. See also \cite{MielkRoub}, Lemma 3.5.\par
		By minimality, going back to the previous estimate, we obtain:
		\begin{align*}
			(\star)&\le\liminf\limits_{n(t)\to+\infty}\left(\mathcal{E}[\widehat{\bm{u}}^{n(t)},\widehat{\bm{\alpha}}^{n(t)}]+\mathcal D[\widehat{\bm{\alpha}}^{n(t)}]
			+\mathcal{K}[\widetilde\delta,\delta_h^{n(t)}(t)\vee\widetilde\delta]\right)\\
			&=\mathcal{E}[{\widetilde{\bm u}},{\widetilde{\bm\alpha}}]+\mathcal D[{\widetilde{\bm\alpha}}]+\mathcal{K}[\widetilde\delta,\gamma(t)\vee\widetilde\delta],
		\end{align*}
		where in the last equality we exploited the strong convergence of $\widehat{\bm{u}}^{n(t)}$ and $\widehat{\bm{\alpha}}^{n(t)}$ towards $\widetilde{\bm{u}}$ and
		$\widetilde{\bm{\alpha}}$, plus assumption ($\varphi$3). Thus we conclude.
	\end{proof}
	To show the validity of (EB') we prove separately the two inequalities. The first one follows from the discrete energy inequality presented in Lemma~\ref{discreteenergyineq}:
	\begin{prop}[\textbf{Upper Energy Estimate}]
		Assume $E_i$ satisfies \eqref{Ecoerc}, $w_i$ satisfies \eqref{wi}, and $\varphi$ satisfies ($\varphi$1), ($\varphi$2). Then for every $t\in[0,T]$ the limit functions
		$\bm{u}$, $\bm{\alpha}$, $\gamma$ obtained in Proposition~\ref{limits} satisfy the following inequality:
		\begin{align*}
			\mathcal{E}[\bm{u}(t),\bm{\alpha}(t)]+\mathcal D[\bm{\alpha}(t)]+\mathcal{K}[\delta(t),\gamma(t)]\le\mathcal{E}[\bm{u}^0,\bm{\alpha}^0]+\mathcal D[\bm{\alpha}^0]
			+\mathcal{K}[\delta^0,\delta^0]+\mathcal{W}[\bm{u},\bm{\alpha}](t).
		\end{align*}
	\end{prop}
	\begin{proof}
		We fix $t\in[0,T]$ and we again consider the subsequence $n(t)$ obtained in Proposition~\ref{limits} (see also Remark~\ref{subsequencermk}); by lower semicontinuity of the energy and Lemma~\ref{discreteenergyineq} we deduce:
		\begin{align*}
			&\quad\mathcal{E}[\bm{u}(t),\bm{\alpha}(t)]+\mathcal D[\bm{\alpha}(t)]+\mathcal{K}[\delta(t),\gamma(t)]\\
			\le& \liminf\limits_{n(t)\to+\infty}\left(\mathcal{E}[\bm{u}^{n(t)}(t),\bm{\alpha}^{n(t)}(t)]+\mathcal D[\bm{\alpha}^{n(t)}(t)]
			+\mathcal{K}[\delta^{n(t)}(t),\delta_h^{n(t)}(t)]\right)\\
			\le& \mathcal{E}[{\bm{u}^0},{\bm{\alpha}^0}]+\mathcal{D}[{\bm{\alpha}^0}]+\mathcal{K}[{\delta}^0,\delta^0]+\limsup\limits_{n(t)\to+\infty}\int_{0}^{t}W^{n(t)}(\tau)\d\tau.
		\end{align*}
		By means of the reverse Fatou's Lemma (we recall that the whole sequence $W^n$ is bounded from above by $C|\dot{\bar{u}}(\tau)|$) we thus get:
		\begin{equation*}
			\limsup\limits_{n(t)\to+\infty}\int_{0}^{t}W^{n(t)}(\tau)\d\tau\le \int_{0}^{t}\limsup\limits_{n(t)\to+\infty} W^{n(t)}(\tau) \d\tau=:(*).
		\end{equation*}
		In order to deal with $(*)$ we argue as follows (see also \cite{CagnToad}, Section 4). We consider the subsequence $n$ (independent of time) obtained in Proposition~\ref{limits} (see also Remark~\ref{subsequencermk}) and for every $\tau\in[0,T]$ we first set
		\begin{equation}\label{Wdef}
			W(\tau):=\limsup\limits_{n\to +\infty} W^{n}(\tau),
		\end{equation}
		which belongs to $L^1(0,T)$ since we recall that $|W^{n}(\tau)|\le C|\dot{\bar{u}}(\tau)|$. Without loss of generality we can assume that the time--dependent subsequences further obtained in Proposition~\ref{limits} also satisfy
		$$W(\tau)=\lim\limits_{n(\tau)\to +\infty} W^{n(\tau)}(\tau),\quad\text{ for every }\tau\in[0,T].$$
		Thus exploiting (a) and (b) in Proposition~\ref{limits} for a.e. $\tau\in[0,T]$ we obtain:
		\begin{equation}\label{Wexpr}
			\begin{aligned}
				W(\tau)&=\lim\limits_{n(\tau)\to +\infty}W^{n(\tau)}(\tau)=\lim\limits_{n(\tau)\to
					+\infty}\frac{\dot{\bar{u}}(\tau)}{L}\sum_{i=1}^{2}\int_{0}^{L}E_i(\alpha^{n(\tau)}_i(\tau,x))(u_i^{n(\tau)})'(\tau,x)\d x\\
				&=\frac{\dot{\bar{u}}(\tau)}{L}\sum_{i=1}^{2}\int_{0}^{L}E_i(\alpha_i(\tau,x))(u_i)'(\tau,x)\d x.
			\end{aligned}
		\end{equation}
		Combining \eqref{Wdef} and \eqref{Wexpr} we finally get $$(*)\le\int_{0}^{t}W(\tau)\d\tau=\mathcal{W}[\bm{u},\bm{\alpha}](t),$$
		and we conclude.
	\end{proof}
	The opposite inequality is instead a byproduct of the global stability condition we proved in Proposition~\ref{globstab}:
	\begin{prop}[\textbf{Lower Energy Estimate}]\label{LEE}
		Assume $E_i$ satisfies \eqref{Ecoerc}, $w_i$ satisfies \eqref{wi}, and $\varphi$ satisfies ($\varphi$1)--($\varphi$3).
		Assume the initial data $\bm{u}^0$, $\bm{\alpha}^0$ fulfil the stability condition \eqref{GS0}.
		Then for every $t\in[0,T]$ the limit functions $\bm{u}$, $\bm{\alpha}$, $\gamma$ obtained in Proposition~\ref{limits} satisfy:
		\begin{align*}
			\mathcal{E}[\bm{u}(t),\bm{\alpha}(t)]+\mathcal D[\bm{\alpha}(t)]+\mathcal{K}[\delta(t),\gamma(t)]\ge\mathcal{E}[\bm{u}^0,\bm{\alpha}^0]
			+\mathcal D[\bm{\alpha}^0]+\mathcal{K}[\delta^0,\delta^0]+\mathcal{W}[\bm{u},\bm{\alpha}](t).
		\end{align*}
	\end{prop}
	\begin{proof}
		If $t=0$ the inequality is trivial, so we fix $t\in(0,T]$ and we consider a sequence of partitions of $[0,t]$ of the form $0=s^n_0<s^n_1<\dots<s^n_{n}=t$ (we stress that this sequence of partitions is completely unrelated with the one considered at the beginning of Subsection~\ref{subseq1} and used to perform the time-discretisation argument) satisfying:
		\begin{itemize}
			\item [(i)] $\displaystyle\lim\limits_{n\to +\infty}\max\limits_{k=1,\dots,n}\left|s^n_{k}-s^n_{k-1}\right|=0$;
			\item[(ii)] $\displaystyle\lim\limits_{n\to +\infty}\sum_{k=1}^{n}\left|(s^n_{k}-s^n_{k-1})\dot{\bar{u}}(s^n_{k})-\int_{s^n_{k-1}}^{s^n_{k}}\dot{\bar{u}}(\tau)\d\tau\right|=0$;
			\item[(iii)] $\displaystyle\lim\limits_{n\to +\infty}\sum_{k=1}^{n}(s^n_{k}-s^n_{k-1})W(s^n_k)=\mathcal{W}[\bm{u},\bm{\alpha}](t)$,
		\end{itemize}
		where $W$ is the function introduced in \eqref{Wdef} and \eqref{Wexpr}. The existence of such a sequence of partitions follows from Lemma 4.5 in \cite{FrancMielk},
		since both $\dot{\bar{u}}$ and $W$ belong to $L^1(0,T)$. In particular, by (i) and the absolute continuity of the integral, we can assume without loss of generality that:
		\begin{itemize}
			\item [(iv)] for every $n\in\enne$ it holds $\displaystyle\int_{s^n_{k-1}}^{s^n_{k}}|\dot{\bar{u}}(\tau)|\d\tau \le \frac 1n$ for every $k=1,\dots,n$.
		\end{itemize}
		For a given partition we fix $k=1,\dots, n$ and, recalling Proposition~\ref{globstab}, we choose as competitors for $\bm{u}(s^n_{k-1})$, $\bm{\alpha}(s^n_{k-1})$ and
		$\gamma(s^n_{k-1})$ in (GS') the functions $\widetilde{\bm{u}}$, $\widetilde{\bm{\alpha}}$, with components:
		\begin{equation*}
			\widetilde{u}_i(x)=u_i(s^n_{k},x)+(\bar u(s^n_{k-1})-\bar u(s^n_{k}))\frac{x}{L},\quad\quad\widetilde{\alpha}_i=\alpha_i(s^n_{k}),\quad\text{ for }i=1,2.
		\end{equation*}
		Recalling that $\gamma(s^n_{k-1})\vee\delta(s^n_k)\le\gamma(s^n_{k})$, and hence
		$\mc K[{\delta(s^n_k)},\gamma(s^n_{k-1})\vee\delta(s^n_k)]\le\mc K[{\delta(s^n_k)},\gamma(s^n_{k})]$ by ($\varphi$3), arguing as in the proof of
		Lemma~\ref{discreteenergyineq} we thus deduce:
		\begin{align*}
			&\mathcal{E}[\bm{u}(s^n_{k-1}),\bm{\alpha}(s^n_{k-1})]{+}\mathcal{D}[\bm{\alpha}(s^n_{k-1})]
			{+}\mathcal{K}[{\delta}(s^n_{k-1}),\gamma(s^n_{k-1})]{-}\mathcal{E}[\bm{u}(s^n_{k}),\bm{\alpha}(s^n_{k})]{-}\mathcal{D}[\bm{\alpha}(s^n_{k})]{-}\mathcal{K}[{\delta}(s^n_{k}),\gamma(s^n_{k})]\\
			&\le-\int_{s^n_{k-1}}^{s^n_k}\frac{\dot{\bar{u}}(\tau)}{L}\sum_{i=1}^{2}\int_{0}^{L}E_i(\alpha_i(s^n_k,x))\left(u_i'(s^n_k,x)
			+\frac{\bar{u}(\tau)-\bar{u}(s^n_k)}{L}\right)\d x \d\tau.
		\end{align*}
		Summing the above inequality from $k=1$ to $k=n$ we obtain:
		\begin{align*}
			&\quad\, \mathcal{E}[\bm{u}(t),\bm{\alpha}(t)]+\mathcal D[\bm{\alpha}(t)]+\mathcal{K}[\delta(t),\gamma(t)]-\mathcal{E}[\bm{u}^0,\bm{\alpha}^0]
			-\mathcal D[\bm{\alpha}^0]-\mathcal{K}[\delta^0,\delta^0]\\
			&\ge \sum_{k=1}^{n}\int_{s^n_{k-1}}^{s^n_{k}}\frac{\dot{\bar{u}}(\tau)}{L}\int_{0}^{L}\sum_{i=1}^2 E_i (\alpha_i(s^n_k,x))\left(u_i'(s^n_{k},x)+\frac{\bar u(\tau)
				-\bar u(s^n_{k})}{L}\right)\d x\d\tau=:J_n
		\end{align*}
		Now we easily notice that $J_n$ can be written as:
		\begin{align*}
			J_n&=\sum_{k=1}^{n}(s^n_k-s^n_{k-1})W(s^n_k)\\
			&+\sum_{k=1}^{n}\int_{s^n_{k-1}}^{s^n_{k}}\frac{\dot{\bar{u}}(\tau)-\dot{\bar{u}}(s^n_{k})}{L}\d\tau\int_{0}^{L}\sum_{i=1}^2 E_i (\alpha_i(s^n_k,x))u_i'(s^n_{k},x)\d x\\
			&+\sum_{k=1}^{n}\int_{s^n_{k-1}}^{s^n_{k}}\frac{\dot{\bar{u}}(\tau)}{L}\frac{\bar u(\tau)-\bar u(s^n_{k})}{L}\d\tau\int_{0}^{L}\sum_{i=1}^2 E_i (\alpha_i(s^n_k,x))\d x
			=:J^1_n+J^2_n+J^3_n.
		\end{align*}
		By (iii) we know that $\lim\limits_{n\to+\infty}J^1_n=\mathcal{W}[\bm{u},\bm{\alpha}](t)$, so we conclude if we prove that
		$\lim\limits_{n\to+\infty}J^2_n=\lim\limits_{n\to+\infty}J^3_n=0$. With this aim we estimate:
		\begin{align*}
			|J^2_n|&\le C\sum_{k=1}^{n}\left|\int_{s^n_{k-1}}^{s^n_{k}}(\dot{\bar{u}}(\tau)-\dot{\bar{u}}(s^n_{k}))\d\tau\right|\left(\sum_{i=1}^2\Vert u_i(s^n_k)\Vert_{H^1(0,L)}\right)\\
			&\le C\sum_{k=1}^{n}\left|(s^n_{k}-s^n_{k-1})\dot{\bar{u}}(s^n_{k})-\int_{s^n_{k-1}}^{s^n_{k}}\dot{\bar{u}}(\tau)\d\tau\right|,
		\end{align*}
		which goes to $0$ by (ii). As regards $J^3_n$, by using (iv) we get:
		\begin{align*}
			|J^3_n|&\le C\sum_{k=1}^{n}\int_{s^n_{k-1}}^{s^n_{k}}|\dot{\bar{u}}(\tau)||\bar{u}(\tau)-\bar{u}(s^n_k)|\d\tau
			=C\sum_{k=1}^{n}\int_{s^n_{k-1}}^{s^n_{k}}|\dot{\bar{u}}(\tau)|\left|\int_{\tau}^{s^n_{k}}\dot{\bar{u}}(s)\d s\right|\d\tau\\
			&\le C\sum_{k=1}^{n}\left(\int_{s^n_{k-1}}^{s^n_{k}}|\dot{\bar{u}}(\tau)|\d\tau\right)^2
			\le \frac{C}{n}\sum_{k=1}^{n}\int_{s^n_{k-1}}^{s^n_{k}}|\dot{\bar{u}}(\tau)|\d\tau=\frac{C}{n}\Vert\dot{\bar{u}}\Vert_{L^1(0,t)},
		\end{align*}
		and the proof is complete.
	\end{proof}
	Putting together what we obtained in this section we infer our first result of existence of generalised energetic evolutions:
	\begin{thm}[\textbf{Existence of Generalised Energetic Evolutions}]\label{exgenensol}
		Let the prescribed displacement $\bar{u}$ belong to $ AC([0,T])$ and the initial data $\bm{u}^0$, $\bm{\alpha}^0$ fulfil \eqref{compatibility} together with the stability
		condition \eqref{GS0}. Assume $E_i$ satisfies \eqref{Ecoerc}, $w_i$ satisfies \eqref{wi}, and $\varphi$ satisfies ($\varphi$1)--($\varphi$3).
		Then the triplet composed by the functions $\bm{u}$, $\bm{\alpha}$ and $\gamma$ obtained in Proposition~\ref{limits} is a generalised energetic evolution.
	\end{thm}
	We conclude this section by showing that, assuming in addition the specific condition ($\varphi$4), which we rewrite also here for the sake of clarity:
	\begin{itemize}
		\item[($\varphi$4)] there exist two functions  $\varphi_1,\varphi_2\colon [0,+\infty)\to[0,+\infty)$ such that $\varphi_1$ is lower semicontinuous, $\varphi_2$ is bounded,
		non-decreasing and concave, and  $\varphi(y,z)=\varphi_1(y)+\varphi_2(z)$,
	\end{itemize}
	the functions $\bm{u}$ and $\bm{\alpha}$ obtained in Proposition~\ref{limits} are actually an energetic evolution. The approach is exactly the same of \cite{CagnToad}.
	We recall that ($\varphi$4) implies ($\varphi$1), ($\varphi$2) and ($\varphi$3).\par 
	We however point out again that ($\varphi$4) does not include most of the cases of loading--unloading cohesive densities $\varphi$ usually arising and adopted in real world applications, like for instance the one presented in Remark~\ref{rmkexamplephi}. The analogous result of Theorem~\ref{exensol} for more realistic densities from the physical point of view is obtained in our main result, contained in Theorem~\ref{finalthm}, via an alternative strategy developed in the forthcoming sections.
	\begin{thm}\label{exensol}
		Let the prescribed displacement $\bar{u}$ belong to $ AC([0,T])$ and the initial data $\bm{u}^0$, $\bm{\alpha}^0$ fulfil \eqref{compatibility} together with the stability
		condition \eqref{GS0}. Assume $E_i$ satisfies \eqref{Ecoerc}, $w_i$ satisfies \eqref{wi}, and $\varphi$ satisfies ($\varphi$4). Then the pair $(\bm{u}, \bm{\alpha})$ obtained
		in Proposition~\ref{limits} is an energetic evolution.\par
		If in addition $\varphi_2$ is strictly increasing, then the function $\gamma$ obtained in Proposition~\ref{limits} coincides with the history variable $\delta_h$.
	\end{thm}
	\begin{proof}
		Thanks to Theorem~\ref{exgenensol} we only need to show the validity of (GS) and (EB) in Definition~\ref{Enersol}. We first focus on (GS); so we fix $t\in[0,T]$ and two
		functions ${\widetilde{\bm u}}\in [H^1_{0,\bar{u}(t)}(0,L)]^2$, ${\widetilde{\bm\alpha}}\in [H^1(0,L)]^2$ such that ${\alpha}_i(t)\le{\widetilde{\alpha}_i}\le {1}$ in $[0,L]$
		for $i=1,2$. Since the triplet $(\bm{u}, \bm{\alpha},\gamma)$ satisfies (GS') we know that:
		\begin{equation*}
			\mathcal{E}[\bm{u}(t),\bm{\alpha}(t)]+\mathcal D[\bm{\alpha}(t)]\le\mathcal{E}[{\widetilde{\bm u}},{\widetilde{\bm\alpha}}]+\mathcal D[{\widetilde{\bm\alpha}}]
			+\mathcal{K}[\widetilde\delta,\gamma(t)\vee\widetilde\delta]-\mathcal{K}[\delta(t),\gamma(t)],
		\end{equation*}
		thus we conclude if we prove
		\begin{equation}\label{claim}
			\mathcal{K}[\widetilde\delta,\gamma(t)\vee\widetilde\delta]-\mathcal{K}[\delta(t),\gamma(t)]
			\le \mathcal{K}[\widetilde\delta,\delta_h(t)\vee\widetilde\delta]-\mathcal{K}[\delta(t),\delta_h(t)].
		\end{equation}
		With this aim, exploiting ($\varphi$4), in particular the monotonicity and concavity of $\varphi_2$, and  recalling that $\gamma(t)\ge \delta_h(t)$, we get:
		\begin{align*}
			\varphi_2(\gamma(t)\vee\widetilde{\delta})&=\varphi_2(\gamma(t)+[\widetilde{\delta}-\gamma(t)]^+)\le \varphi_2(\gamma(t)+[\widetilde{\delta}-\delta_h(t)]^+)\\
			&\le \varphi_2(\gamma(t))+\varphi_2(\delta_h(t)+[\widetilde{\delta}-\delta_h(t)]^+)-\varphi_2(\delta_h(t))\\
			&=\varphi_2(\gamma(t))+\varphi_2(\delta_h(t)\vee\widetilde{\delta})-\varphi_2(\delta_h(t)).
		\end{align*}
		The above inequality implies:
		\begin{align*}
			\mathcal{K}[\widetilde\delta,\gamma(t)\vee\widetilde\delta]-\mathcal{K}[\widetilde\delta,\delta_h(t)\vee\widetilde{\delta}]
			&=\int_{0}^{L}\Big(\varphi_2(\gamma(t,x)\vee\widetilde\delta(x))-\varphi_2(\delta_h(t,x)\vee\widetilde\delta(x))\Big)\d x\\
			&\le \int_{0}^{L}\Big(\varphi_2(\gamma(t,x))-\varphi_2(\delta_h(t,x))\Big)\d x\\
			&=\mathcal{K}[\delta(t),\gamma(t)]-\mathcal{K}[\delta(t),\delta_h(t)],
		\end{align*}
		which is equivalent to \eqref{claim}.\par
		We now prove (EB). Since the triplet $(\bm{u}, \bm{\alpha},\gamma)$ satisfies (EB'), it is enough to prove
		\begin{equation}\label{claim2}
			\mc K[\delta(t),\gamma(t)]=\mc K[\delta(t),\delta_h(t)],\quad\text{ for every }t\in[0,T].
		\end{equation}
		Since $\gamma(t)\ge \delta_h(t)$ we easily deduce $\mc K[\delta(t),\gamma(t)]\ge\mc K[\delta(t),\delta_h(t)]$. To get the other inequality we first observe
		that arguing exactly as in the proof of Proposition~\ref{LEE}, but replacing $\gamma$ with $\delta_h$ (indeed we have just proved (GS)) we get:
		\begin{align*}
			\mathcal{E}[\bm{u}(t),\bm{\alpha}(t)]+\mathcal D[\bm{\alpha}(t)]+\mathcal{K}[\delta(t),\delta_h(t)]\ge\mathcal{E}[\bm{u}^0,\bm{\alpha}^0]
			+\mathcal D[\bm{\alpha}^0]+\mathcal{K}[\delta^0,\delta^0]+\mathcal{W}[\bm{u},\bm{\alpha}](t).
		\end{align*}
		Combining the above inequality with (EB') we finally obtain
		\begin{equation*}
			\mc K[\delta(t),\delta_h(t)]\ge\mc K[\delta(t),\gamma(t)],
		\end{equation*}
		hence \eqref{claim2} holds true.\par
		If in addition $\varphi_2$ is strictly increasing, then \eqref{claim2} implies $\gamma(t)=\delta_h(t)$ since both functions are continuous in $[0,L]$. Thus we conclude.
	\end{proof}
	
	\section{PDE Form of Energetic Evolutions}\label{seceulerlagrange}
	In this section we compute the Euler--Lagrange equations coming from the global stability condition (GS'). More precisely we prove that any generalised energetic
	evolution $(\bm{u}, \bm{\alpha},\gamma)$ must satisfy, in a suitable weak formulation, the following system of equilibrium equations governing the stresses $\sigma_i$
	(see Proposition~\ref{Propelu}): 
	\begin{subequations}
		\begin{equation}\label{systemstresses}
			\begin{cases} 
				-\sigma_1(t)'+\partial_y\varphi(\delta(t),\gamma(t))\sgn(u_1(t)-u_2(t))=0,&\text{ in }[0,L],\\
				-\sigma_2(t)'-\partial_y\varphi(\delta(t),\gamma(t))\sgn(u_1(t)-u_2(t))=0,& \text{ in }[0,L],
			\end{cases}
			\qquad\text{ for every }t\in [0,T],
		\end{equation}
		where $\sgn(\cdot)$ denotes the signum function, together with a Karush--Kuhn--Tucker condition describing the evolution of the damage variables (if regular in time, see Propositions~\ref{Propelalpha} and \ref{propimportant}):
		\begin{equation}\label{KKT}
			\begin{cases}
				\dot{\alpha}_i(t)\ge 0,&\text{ in }[0,L],\\
				-\alpha_i(t)''+\frac 12 E_i'(\alpha_i(t))(u_i(t)')^2+w_i'(\alpha_i(t))\ge 0,&\text{ in }[0,L],\\
				\Big[-\alpha_i(t)''+\frac 12 E_i'(\alpha_i(t))(u_i(t)')^2+w_i'(\alpha_i(t))\Big]\dot{\alpha}_i(t)= 0,&\text{ in }[0,L],
			\end{cases}
			\qquad\text{ for a.e. }t\in [0,T].
		\end{equation}
	\end{subequations}
	The results of this section will be crucial for the achievement of our goal, namely the equivalence between the fictitious history variable $\gamma$ and the concrete one
	$\delta_h$, under meaningful assumptions on $\varphi$. The argument based on temporal regularity of generalised energetic evolutions will be developed in
	Section~\ref{sectemporalregularity}.\par
	We recall that, given the loading--unloading density $\varphi\colon \mc T\to [0,+\infty)$, we denote by $\psi$ its restriction to the diagonal, namely $\psi(z)=\varphi(z,z)$,
	for $z\in [0,+\infty)$. Throughout the section the main assumptions on $\varphi$ (and $\psi$) are:
	\begin{subequations}\label{secondassumptions}
		\begin{equation}\label{2a}
			\text{the function }\psi \text{ belongs to }C^1([0,+\infty));
		\end{equation}
		\begin{equation}\label{2b}
			\text{for every }z\in(0,+\infty)\text{ the map }\varphi(\cdot,z)\text{ belongs to }C^1([0,z]);
		\end{equation}
		\begin{equation}\label{2c}
			\text{ for every }z\in(0,+\infty)\text{ there holds }\partial_y\varphi(z,z)=\psi'(z)\text{ and }\partial_y\varphi(0,z)=0;
		\end{equation}
		\begin{equation}\label{2d}
			\text{ the partial derivative }\partial_y\varphi\text{ belongs to } C^0(\mc T\setminus(0,0)) \text{ and it is bounded in }\mc T.
		\end{equation}
	\end{subequations}
	We notice that the above conditions are slightly more general than properties ($\varphi$5)--($\varphi$8) listed in Section~\ref{secsetting}, since we do not require
	any convexity assumption (which will be instead employed in Section~\ref{sectemporalregularity}).\par
	We start the analysis with a simple but useful lemma.
	\begin{lemma}\label{Lemmader}
		Let $f,g\in \erre$ such that $f\ge|g|$ and assume the function $\varphi\colon \mc T\to [0,+\infty)$ satisfies: 
		\begin{subequations}
			\begin{equation}
				\text{the function }z\mapsto \varphi(z,z)=:\psi(z) \text{ is differentiable in }[0,+\infty);
			\end{equation}
			\begin{equation}
				\text{for every }z\in(0,+\infty)\text{ the map }\varphi(\cdot,z)\text{ is differentiable in }[0,z];
			\end{equation}
			\begin{equation}
				\text{for every }z\in(0,+\infty)\text{ there holds }\partial_y\varphi(z,z)=\psi'(z).
			\end{equation}
		\end{subequations}
		Then for every $v\in \erre$ one has:
		\begin{equation*}
			\lim\limits_{h\to 0^+}\frac{\varphi(|g+hv|,f\vee|g+hv|)-\varphi(|g|,f)}{h}=\begin{cases}
				\partial_y\varphi(|g|,f)\sgn(g)v, &\text{ if }f>|g|>0,\\
				\psi'(|g|)\sgn (g)v, &\text{ if }f=|g|>0,\\
				\partial_y\varphi(0,f)|v|, &\text{ if }f>|g|=0,\\
				\psi'(0)|v|, & \text{ if }f=|g|=0.
			\end{cases}
		\end{equation*}
	\end{lemma}
	\begin{proof}
		We denote by $I$ the limit we want to compute and we distinguish among all the different cases. We first assume that $f>|g|$, so we get:
		\begin{itemize}
			\item if $g=0$, then $\displaystyle I=\lim\limits_{h\to 0^+}\frac{\varphi(h|v|,f)-\varphi(0,f)}{h}=\partial_y\varphi(0,f)|v|$;
			\item if $g>0$, then $\displaystyle I=\lim\limits_{h\to 0^+}\frac{\varphi(g+hv,f)-\varphi(g,f)}{h}=\partial_y\varphi(g,f)v=\partial_y\varphi(|g|,f)\sgn(g)v$;
			\item if $g<0$, then $\displaystyle I=\lim\limits_{h\to 0^+}\frac{\varphi(|g|-hv,f)-\varphi(|g|,f)}{h}=-\partial_y\varphi(|g|,f)v=\partial_y\varphi(|g|,f)\sgn(g)v$.
		\end{itemize}
		If instead $f=|g|$ we have:
		\begin{itemize}
			\item if $g=0$, then $\displaystyle I=\lim\limits_{h\to 0^+}\frac{\varphi(h|v|,h|v|)-\varphi(0,0)}{h}=\lim\limits_{h\to 0^+}\frac{\psi(h|v|)-\psi(0)}{h}=\psi'(0)|v|$;
			\item if $g>0$ and $v\ge 0$, then $\displaystyle I=\lim\limits_{h\to 0^+}\frac{\varphi(g+hv,g+hv)-\varphi(g,g)}{h}=\psi'(g)v=\psi'(|g|)\sgn(g)v$;
			\item if $g>0$ and $v<0$, then $\displaystyle I=\lim\limits_{h\to 0^+}\frac{\varphi(g+hv,g)-\varphi(g,g)}{h}=\partial_y\varphi(g,g)v=\psi'(g)v=\psi'(|g|)\sgn(g)v$;
			\item if $g<0$ and $v\ge 0$, then $\displaystyle I=\lim\limits_{h\to 0^+}\frac{\varphi(|g|-hv,|g|)-\varphi(|g|,|g|)}{h}=-\partial_y\varphi(|g|,|g|)v=\psi'(|g|)\sgn(g)v$;
			\item if $g<0$ and $v< 0$, then $\displaystyle I=\lim\limits_{h\to 0^+}\frac{\psi(|g|-hv)-\psi(|g|)}{h}=-\psi'(|g|)v=\psi'(|g|)\sgn(g)v$.
		\end{itemize}
		So we conclude.
	\end{proof}
	As an immediate corollary we deduce:
	\begin{cor}\label{passageint}
		Let $f,g$ be two measurable functions such that $f\in L^\infty(0,L)$ and $f\ge|g|$ a.e. in $[0,L]$, and assume $\varphi$ satisfies \eqref{2a}, \eqref{2b} and \eqref{2c}.
		Then for every $v\in L^\infty(0,L)$ it holds:
		\begin{equation*}
			\lim\limits_{h\to 0^+}\frac{\mc K[|g+hv|,f\vee|g+hv|]{-}\mc K[|g|,f]}{h}=\!\!\int_{\{|g|>0\}}\!\!\!\!\!\!\!\!\!\!\!\!\!\!\!\partial_y\varphi(|g(x)|,f(x))\sgn(g(x))v(x)\d x
			\,+\,\psi'(0)\!\!\int_{\{f=0\}}\!\!\!\!\!\!\!\!\!\!\!\!\!|v(x)|\d x.
		\end{equation*}
	\end{cor}
	\begin{proof}
		We notice that, by the explicit expression of $\mc K$ given by \eqref{Cohesiveenergy}, the limit we want to compute can be written as
		\begin{equation*}
			\lim\limits_{h\to 0^+}\int_{0}^{L}\frac{\varphi(|g(x)+hv(x)|,f(x)\vee|g(x)+hv(x)|)-\varphi(|g(x)|,f(x))}{h}\d x.
		\end{equation*}
		Assumptions \eqref{2a} and \eqref{2b} allow us to pass to the limit inside the integral, thus we conclude by means of Lemma~\ref{Lemmader} and exploiting \eqref{2c}.
	\end{proof}
	We are now in a position to state and prove the first result of this section, namely a weak form of the Euler--Lagrange equation for the displacement $\bm{u}$, or better
	for the stress $\bm{\sigma}$.
	\begin{prop}\label{Propelu}
		Let $E_i\in C^0([0,1])$ and assume $\varphi$ satisfies \eqref{2a}, \eqref{2b} and \eqref{2c}.
		Let $(\bm{u},\bm{\alpha},\gamma)$ satisfy (CO') and (GS') of
		Definition~\ref{Generalisedenersol}. 
		Then for every $t\in [0,T]$ and for every $\bm{v}\in[H^1_0(0,L)]^2$ it holds:
		\begin{equation}\label{ELu}
			\left |\int_{0}^{L}\sum_{i=1}^{2}\sigma_i(t)v_i'\d x+\int_{\{\delta(t)>0\}}\!\!\!\!\!\!\!\!\!\!\!\big[\partial_y\varphi(\delta(t),\gamma(t))\sgn(u_1(t){-}u_2(t))\big](v_1-v_2)\d x\right |
			\le \psi'(0) \int_{\{\gamma(t)=0\}}\!\!\!\!\!\!\!\!\!\!\!\!|v_1-v_2|\d x,
		\end{equation}
		where the stresses $\sigma_i$ have been introduced in \eqref{stresses}.\\
		In particular, for every $t\in[0,T]$ the sum of the stresses $\displaystyle\sum_{i=1}^{2}\sigma_i(t)$ is constant in $[0,L]$.
	\end{prop}
	\begin{proof}
		We fix $t\in[0,T]$ and by choosing $\widetilde{\bm{\alpha}}=\bm{\alpha}(t)$ in (GS') we get for every $h>0$ and $\bm{v}\in [H^1_0(0,L)]^2$:
		\begin{align*}
			&\quad\,\,\mc E[\bm{u}(t),\bm{\alpha}(t)]+\mc K[\delta(t),\gamma(t)]\\
			&\le\mc E[\bm{u}(t)+h\bm{v},\bm{\alpha}(t)]
			+\mc K[|u_1(t)-u_2(t)+h(v_1-v_2)|,\gamma(t)\vee|u_1(t)-u_2(t)+h(v_1-v_2)|].
		\end{align*}
		Letting $h\to 0^+$ we thus deduce
		\begin{align*}
			0&\le \lim\limits_{h\to 0^+}\frac {\mc E[\bm{u}(t)+h\bm{v},\bm{\alpha}(t)]-\mc E[\bm{u}(t),\bm{\alpha}(t)]}{h}\\
			&\quad+\lim\limits_{h\to 0^+}\frac{\mc K[|u_1(t)-u_2(t)+h(v_1-v_2)|,\gamma(t)\vee|u_1(t)-u_2(t)+h(v_1-v_2)|]-K[\delta(t),\gamma(t)]}{h}.
		\end{align*}
		The first limit is trivially equal to $\int_{0}^{L}\sum_{i=1}^{2}\sigma_i(t)v_i'\d x$, while for the second one we employ Corollary~\ref{passageint} and we finally obtain:
		\begin{align*}
			0&\le \int_{0}^{L}\sum_{i=1}^{2}\sigma_i(t)v_i'\d x+\int_{\{\delta(t)>0\}}\big[\partial_y\varphi(\delta(t),\gamma(t))\sgn(u_1(t)-u_2(t))\big](v_1-v_2)\d x\\
			&\quad+ \psi'(0) \int_{\{\gamma(t)=0\}}|v_1-v_2|\d x.
		\end{align*}
		By following the same argument with $-\bm{v}$, we prove \eqref{ELu}.\par
		In particular if $v_1=v_2=:v$ we deduce that
		\begin{equation*}
			\int_{0}^{L}\left(\sum_{i=1}^{2}\sigma_i(t)\right)v'\d x=0,\quad\text{ for every }v\in H^1_0(0,L),
		\end{equation*}
		and so $\sum_{i=1}^{2}\sigma_i(t)$ is constant in $[0,L]$.
	\end{proof}
	We want to point out that if $\psi'(0)$ were equal to $0$ (usually false in a cohesive setting, in which $\psi$ is concave and strictly increasing, see Remark~\ref{rmkexamplephi}), then inequality \eqref{ELu} would actually be equivalent to the system 
	\eqref{systemstresses}. The simplifications brought by the assumption $\psi'(0)=0$ can be also found in \cite{AleFredd1d}, where it has been used for numerical reasons, 
	and in \cite{NegVit}, where it has been exploited to perform an approximation argument.\par
	In our work, however, we do not need that additional (and not reasonable) assumption, indeed inequality \eqref{ELu} will be enough for our purposes.\par
	The next proposition deals with the damage variable $\bm{\alpha}$:
	\begin{prop}\label{Propelalpha}
		Assume $E_i, w_i\in C^1([0,1])$ and let $(\bm{u},\bm{\alpha},\gamma)$ satisfy (CO') and (GS') of Definition~\ref{Generalisedenersol}. 
		Then, for every $t\in [0,T]$ and for every $\bm{\beta}\in [H^1(0,L)]^2$ such that $\beta_i\ge 0$ for $i=1,2$, it holds:
		\begin{equation*}
			\sum_{i=1}^{2}\left(\frac 12 \int_{\{\alpha_i(t)<1\}}E_i'(\alpha_i(t))(u_i(t)')^2\beta_i\d x+\int_{\{\alpha_i(t)<1\}} w_i'(\alpha_i(t))\beta_i\d x
			+\int_{\{\alpha_i(t)<1\}}\alpha_i(t)'\beta_i'\d x\right)\ge 0.
		\end{equation*}
	\end{prop}
	\begin{proof}
		We fix $t\in[0,T]$ and by choosing $\widetilde{\bm{u}}=\bm{u}(t)$ in (GS') we get
		\begin{equation}\label{sta}
			\mc E[\bm{u}(t),\bm{\alpha}(t)]+\mc D[\bm{\alpha}(t)]\le \mc E[\bm{u}(t),\widetilde{\bm{\alpha}}]
			+\mc D[\widetilde{\bm{\alpha}}],\text{ for every }\widetilde{\bm{\alpha}}\in [H^1(0,L)]^2\text{ s.t. }\alpha_i(t)\le\alpha_i\le 1.
		\end{equation}
		We now fix $\bm{\beta}\in [H^1(0,L)]^2$ such that $\beta_i\ge 0$ and given $h>0$ we define $\widetilde{\bm{\alpha}}^h(t,x)$ as the vector in $\erre^2$ whose components 
		are $(\alpha_i(t,x)+h\beta_i(x))\wedge 1$. By plugging $\widetilde{\bm{\alpha}}^h(t)$ in \eqref{sta} as a test function and letting $h\to 0^+$ we thus deduce:
		\begin{equation}\label{est1}
			\begin{aligned}
				0&
				\le \liminf\limits_{h\to 0^+}\frac{\mc E[\bm{u}(t),\widetilde{\bm{\alpha}}^h(t)]-\mc E[\bm{u}(t),\bm{\alpha}(t)]+\mc D[\widetilde{\bm{\alpha}}^h(t)]-\mc D[\bm{\alpha}(t)]}{h}\\
				&=\liminf\limits_{h\to 0^+}\sum_{i=1}^{2}\Big(\frac 12 \int_{0}^{L}\frac{E_i(\widetilde{{\alpha}}^h_i(t))-E_i(\alpha_i(t))}{h}(u_i(t)')^2\d x
				+\int_{0}^{L}\frac{w_i(\widetilde{{\alpha}}^h_i(t))-w_i(\alpha_i(t))}{h}\d x\\&\qquad\qquad\qquad
				+\frac 12\int_{0}^{L}\frac{(\widetilde{{\alpha}}^h_i(t)')^2-(\alpha_i(t)')^2}{h}\d x \Big) = \liminf\limits_{h\to 0^+}  (I_h+ II_h+ III_h).
			\end{aligned}
		\end{equation}
		We study the limits of $I_h, \, II_h, \, III_h$ separately. Since $E_i,w_i$ are in  $C^1([0,1])$ we can pass the limit inside the integral in both $I_h$ and $II_h$. 
		We also notice that given $f\in C^1([0,1])$, $a\in[0,1]$ and $b\ge 0$ one has
		\begin{equation*}
			\lim\limits_{h\to 0^+}\frac{f((a+hb)\wedge 1)-f(a)}{h}=\begin{cases}
				f'(a)b,&\text{ if } a\in[0,1),\\
				0,&\text{ if } a=1.
			\end{cases}
		\end{equation*}
		Thus we deduce that
		\begin{equation}\label{est2}
			\lim\limits_{h\to 0^+}I_h=\sum_{i=1}^{2}\frac 12\int_{\{\alpha_i(t)<1\}}\!\!\!\!\!\!\!\!E_i'(\alpha_i(t))(u_i(t)')^2\beta_i\d x,
			\quad\text{ and }\quad \lim\limits_{h\to 0^+}II_h=\sum_{i=1}^{2}\int_{\{\alpha_i(t)<1\}}\!\!\!\!\!\!\!\! w_i'(\alpha_i(t))\beta_i\d x.
		\end{equation}
		To deal with $III_h$ we first observe that $\widetilde{{\alpha}}^h_i(t)'=\begin{cases}
			\alpha_i(t)'+h\beta_i', &\text{ a.e. in }\{\alpha_i(t)+h\beta_i<1\},\\
			0,&\text{ a.e. in }\{\alpha_i(t)+h\beta_i\ge 1\},
		\end{cases}$ and so
		\begin{align*}
			III_h&=\sum_{i=1}^{2}\left(\int_{\{\alpha_i(t)+h\beta_i<1\}}\!\!\alpha_i(t)'\beta_i'\d x+\frac h2\int_{\{\alpha_i(t)+h\beta_i<1\}}\!\!(\beta_i')^2\d x
			-\frac{1}{2h}\int_{\{\alpha_i(t)+h\beta_i\ge1\}}\!\!(\alpha_i(t)')^2\d x\right)\\
			&\le \sum_{i=1}^{2}\left(\int_{\{\alpha_i(t)+h\beta_i<1\}}\alpha_i(t)'\beta_i'\d x+\frac h2\int_{0}^{L}(\beta_i')^2\d x\right).
		\end{align*}
		By an easy application of dominated convergence theorem we hence obtain
		\begin{equation}\label{est3}
			\limsup\limits_{h\to 0^+}III_h\le \sum_{i=1}^{2}\int_{\{\alpha_i(t)<1\}}\alpha_i(t)'\beta_i'\d x,
		\end{equation}
		and collecting \eqref{est1}, \eqref{est2} and \eqref{est3} we conclude.
	\end{proof}
	The last result of the section is a byproduct of the energy balance (EB'), assuming a priori that a generalised energetic evolution possesses a certain regularity in time. 
	This kind of regularity will be however proved in Section~\ref{sectemporalregularity} under suitable convexity assumptions on the data, thus this a priori requirement is 
	not restrictive.\par
	We refer to the Appendix for the definition and the main properties of absolutely continuous functions in Banach spaces, concepts we use in the next proposition.
	\begin{prop}\label{propimportant}
		Assume $E_i,w_i\in C^1([0,1])$ and that $\varphi$ satisfies \eqref{secondassumptions} and ($\varphi$3). Let $(\bm{u},\bm{\alpha},\gamma)$ be a generalised energetic 
		evolution such that:
		\begin{equation*}
			\bm{u},\bm{\alpha}\in AC([0,T];[H^1(0,L)]^2),\quad\text{ and }\quad \gamma\in C^0([0,T], C^0([0,L])).
		\end{equation*}
		Then for a.e. $t\in[0,T]$ one has:
		\begin{align}\label{zeroder}
			&\bullet\frac 12 \int_{0}^{L}\!\!E_i'(\alpha_i(t))(u_i(t)')^2\dot{\alpha}_i(t)\d x+\int_{0}^{L} \!\!w_i'(\alpha_i(t))\dot{\alpha}_i(t)\d x
			+\int_{0}^{L}\alpha_i(t)'\dot{\alpha}_i(t)'\d x=0,\quad\text{ for }i=1,2;\nonumber\\
			&\bullet\lim\limits_{h\to 0}\int_{0}^{L}\frac{\varphi(\delta(t),\gamma(t+h))-\varphi(\delta(t),\gamma(t))}{h}\d x=0.
		\end{align}
	\end{prop}
	\begin{proof}
		First of all we notice that the temporal regularity we are assuming on $\bm{u}$ and $\bm{\alpha}$  ensures that the maps $t\mapsto \mc E[\bm{u}(t), \bm{\alpha}(t)]$ 
		and $t\mapsto \mc D[\bm{\alpha}(t)]$ are absolutely continuous in $[0,T]$. Moreover for almost every time $t\in[0,T]$ the following expressions for their derivatives can 
		be easily obtained:
		\begin{subequations}\label{derED}
			\begin{equation}
				\!\!\!\frac{\d}{\d t}\mc E[\bm{u}(t), \bm{\alpha}(t)]
				=\sum_{i=1}^{2}\left(\frac 12 \int_{0}^{L}E_i'(\alpha_i(t))(u_i(t)')^2\dot{\alpha}_i(t)\d x+\int_{0}^{L}E_i(\alpha_i(t))u_i(t)'\dot{u}_i(t)'\d x\right);
			\end{equation}
			\begin{equation}
				\frac{\d}{\d t}\mc D[\bm{\alpha}(t)]=\sum_{i=1}^{2}\left(\int_{0}^{L}w_i'(\alpha_i(t))\dot{\alpha}_i(t)\d x+\int_{0}^{L}\alpha_i(t)'\dot{\alpha}_i(t)'\d x\right).
			\end{equation}
		\end{subequations}
		By (EB'), since the work of the prescribed displacement $\mc W[\bm{u},\bm{\alpha}]$ is absolutely continuous by definition, we now deduce that also 
		the map $t\mapsto \mc K[\delta(t),\gamma(t)]$ is absolutely continuous in $[0,T]$. Moreover we know that $\delta$ belongs to $AC([0,T];H^1_0(0,L))$, 
		indeed both $u_1$ and $u_2$ are absolutely continuous with values in $H^1(0,L)$ by assumption. 
		Thus for almost every $t\in [0,T]$ there exists the derivative of $\mc K[\delta(t),\gamma(t)]$ and we can compute:
		\begin{equation}\label{derK}
			\begin{aligned}
				\frac{\d}{\d t}\mc K[\delta(t),\gamma(t)]&=\lim\limits_{h\to 0}\int_{0}^{L}\frac{\varphi(\delta(t+h),\gamma(t+h))-\varphi(\delta(t),\gamma(t))}{h}\d x\\
				&=\int_{0}^{L}\partial_y\varphi(\delta(t),\gamma(t))\dot{\delta}(t)\d x+\lim\limits_{h\to 0}\int_{0}^{L}\frac{\varphi(\delta(t),\gamma(t+h))-\varphi(\delta(t),\gamma(t))}{h}\d x,
			\end{aligned}
		\end{equation}
		where we exploited the continuity assumption of both $\partial_y\varphi$ and $\gamma$.\par
		Differentiating (EB'), using \eqref{derED} and \eqref{derK}, and recalling that the sum of the stresses $\sigma_i$ is constant in $[0,L]$ by Proposition~\ref{Propelu}, 
		we deduce, for almost every $t\in [0,T]$,
		\begin{equation}\label{dereb}
			\begin{aligned}
				0&=\sum_{i=1}^{2}\left(\frac 12 \int_{0}^{L}E_i'(\alpha_i(t))(u_i(t)')^2\dot{\alpha}_i(t)\d x
				+\int_{0}^{L}w_i'(\alpha_i(t))\dot{\alpha}_i(t)\d x+\int_{0}^{L}\alpha_i(t)'\dot{\alpha}_i(t)'\d x\right)\\
				&\quad+\int_{0}^{L}\sum_{i=1}^{2}E_i(\alpha_i(t))u_i(t)'\dot{u}_i(t)'\d x+\int_{0}^{L}\partial_y\varphi(\delta(t),\gamma(t))\dot{\delta}(t)\d x
				-\dot{\bar{u}}(t)\sum_{i=1}^{2}\sigma_i(t,0)\\
				&\quad+\lim\limits_{h\to 0}\int_{0}^{L}\frac{\varphi(\delta(t),\gamma(t+h))-\varphi(\delta(t),\gamma(t))}{h}\d x.
			\end{aligned}
		\end{equation}
		The term in the third line of \eqref{dereb} is nonnegative by means of ($\varphi$3) and the fact that $\gamma$ is non-decreasing (in time). 
		We thus conclude if we show that also the sum of the terms in the second line and each of the two terms (for $i=1,2$) in the sum in the first line are nonnegative. \par
		We first focus on the first line. We notice that for $i=1,2$, the function $\dot{\alpha}_i(t)\in H^1(0,L)$ is nonnegative and vanishes on the set $\{\alpha_i(t)=1\}$; 
		indeed $\alpha_i$ is non-decreasing in time and it is always less or equal than $1$. This means that we can use it as a test function in Proposition~\ref{Propelalpha}, 
		getting for a.e. $t\in [0,T]$:
		\begin{equation*}
			\frac 12 \int_{0}^{L}E_i'(\alpha_i(t))(u_i(t)')^2\dot{\alpha}_i(t)\d x+\int_{0}^{L}w_i'(\alpha_i(t))\dot{\alpha}_i(t)\d x+\int_{0}^{L}\alpha_i(t)'\dot{\alpha}_i(t)'\d x\ge 0.
		\end{equation*}
		As regards the sum of the terms in the second line in \eqref{dereb}, we actually prove it is equal to zero. To this aim we make use of Proposition~\ref{Propelu} 
		choosing as test functions $v_i(x)=\dot{u}_i(t,x)-\dot{\bar{u}}(t)x/L\in H^1_0(0,L)$, so that $|v_1-v_2|=|\dot{u}_1(t)-\dot{u}_2(t)|$. 
		We indeed notice that $|v_1-v_2|=0$ on the set $\{\gamma(t)=0\}$: if $x$ belongs to that set, then $u_1(\tau,x)=u_2(\tau,x)$ for every $\tau\in[0,t]$, 
		and thus $\dot{u}_1(t,x)=\dot{u_2}(t,x)$. So we deduce for a.e. $t\in[0,T]$:
		\begin{align*}
			0&=\int_{0}^{L}\sum_{i=1}^{2}\sigma_i(t)\left(\dot{u}_i(t)'-\frac{\dot{\bar{u}}(t)}{L}\right)\d x
			+\int_{\{\delta(t)>0\}}\!\!\!\!\!\!\!\!\!\!\big[\partial_y\varphi(\delta(t),\gamma(t))\sgn(u_1(t)-u_2(t))\big](\dot{u}_1(t)-\dot{u}_2(t))\d x\\
			&=\int_{0}^{L}\sum_{i=1}^{2}E_i(\alpha_i(t))u_i(t)'\dot{u}_i(t)'\d x+\int_{0}^{L}\partial_y\varphi(\delta(t),\gamma(t))\dot{\delta}(t)\d x
			-\dot{\bar{u}}(t)\sum_{i=1}^{2}\sigma_i(t,0).
		\end{align*}
		In the above equality we first used the fact that by definition 
		$$\dot{\delta}(t)=(\dot{u}_1(t)-\dot{u}_2(t))\sgn(u_1(t)-u_2(t)),\quad\text{ in }\{\delta(t)>0\},$$ 
		and then we exploited the assumption $\partial_y\varphi(0,z)=0$ for $z>0$.\par
		So the proof is complete.
	\end{proof}
	\section{Temporal Regularity and Equivalence between $\gamma$ and $\delta_h$}\label{sectemporalregularity}
	
	In this last section we finally develop the strategy which will allow to show that the fictitious history variable $\gamma$ actually coincides with the concrete one 
	$\delta_h$ in some meaningful cases, see Theorems~\ref{thmgammadeltah} and \ref{finalthm}. The argument, which exploits the results of Section~\ref{seceulerlagrange}, 
	is based on the regularity in time of generalised energetic evolutions; this feature, as noticed in \cite{Thom}, is a peculiarity of systems governed by convex energies. 
	For this reason, in this section we need to strengthen the assumptions on the data, requiring for $i=1,2$:
	\begin{equation}\label{convexE}
		E_i\in C^2([0,1])\text{ is convex and satisfies }\frac 12 E_i''(y)E_i(y)-E_i'(y)^2>0\text{ for every }y\in[0,1];
	\end{equation}
	\begin{equation}\label{convexw}
		\begin{gathered}
			w_i\in C^1([0,1])\text{ satisfies \eqref{wi} and is uniformly convex with parameter }\mu_i>0,\text{ namely }\\
			w_i(\theta y^a+(1-\theta)y^b)\le \theta w_i(y^a)+(1{-}\theta)w_i(y^b)-\frac{\mu_i}{2}\theta(1{-}\theta)|y^a{-}y^b|^2,\text{ for every }\theta,y^a,y^b\in [0,1].
		\end{gathered}
	\end{equation}
	We notice that \eqref{convexE} implies \eqref{Ecoerc}, while \eqref{convexw} is trivially satisfied for instance by the simple example $w_i(y)=\frac{y^2+y}{2}$. We also define
	\begin{equation}
		M_i:=\max\limits_{y\in[0,1]}E_i''(y),\quad\quad m_i:=\min\limits_{y\in[0,1]}\left(\frac 12 E_i''(y)E_i(y)-E_i'(y)^2\right),
	\end{equation}
	which are strictly positive by \eqref{convexE}, and we finally denote by $\mu$ the minimum between $\mu_1$ and $\mu_2$, namely
	\begin{equation}
		\mu:=\mu_1\wedge\mu_2>0.
	\end{equation}
	\begin{rmk}[\textbf{Hardening Materials}]
		Condition \eqref{convexE} is a characteristic of the so called hardening materials, namely those materials for which the compliance $S(y):=E(y)^{-1}$ is strictly concave. 
		Indeed by simple calculations one has:
		\begin{equation*}
			S''(y)=-\frac{2}{E(y)^3}\left(\frac 12 E''(y)E(y)-E'(y)^2\right),
		\end{equation*}
		from which $S''<0$ if and only if \eqref{convexE} is satisfied. Temporal regularity of evolutions is expected only for this kind of materials, indeed in the opposite 
		framework of softening materials (with convex compliance $S$) discontinuous evolutions are common due to snap--back phenomena (see also the analysis of \cite{PhamMarigo}).
	\end{rmk}
	Of course we also need some sort of convexity for the loading--unloading density $\varphi$. However, we recall that usually it originates from a concave function $\psi$, 
	see Remark~\ref{rmkexamplephi}; thus, in order to keep that crucial property, we only require a weak form of convexity assumption on $\psi$, already adopted in \cite{NegSca}:
	\begin{subequations}\label{convexityassumptionsphi}
		\begin{equation}\label{convexpsi}
			\begin{gathered}
				\text{the function }\psi \text{ is }\lambda\text{--convex for some }\lambda>0,\text{ namely for every }\theta\in [0,1]\text{ and }z^a,z^b\in [0,+\infty)\\
				\psi(\theta z^a+(1-\theta)z^b)\le \theta\psi (z^a)+(1-\theta)\psi (z^b)+\frac \lambda 2\theta(1-\theta)|z^a-z^b|^2,
			\end{gathered}
		\end{equation}
		while for $\varphi$ itself, in addition to \eqref{secondassumptions}, we assume:
		\begin{equation}\label{convexphi}
			\text{for every }z\in(0,+\infty)\text{ the map }\varphi(\cdot,z)\text{ is non-decreasing and convex.}
		\end{equation}
	\end{subequations}
	\begin{rmk}
		Coupling \eqref{secondassumptions} with \eqref{convexityassumptionsphi} we have thus recovered the assumptions ($\varphi$5)--($\varphi$8) listed in Section~\ref{secsetting}. 
		We point out again that they are satisfied by the prototypical example of loading--unloading density $\varphi$ given by \eqref{examplephi}.
	\end{rmk}
	A crucial condition on the involved parameters will be given by
	\begin{equation}\label{maincondition}
		\frac{m_1}{M_1}\wedge\frac{m_2}{M_2}>\lambda \frac{L^2}{\pi^2}.
	\end{equation}
	It morally says that the convexity of the internal energy $\mc E$, represented by $\frac{m_1}{M_1}\wedge\frac{m_2}{M_2}$, is stronger than the concavity of $\mc K$, 
	represented by $\lambda$, and thus the overall behaviour is the one of a convex energy.
	\begin{rmk}
		As already observed in \cite{Thom,PhamMarigo}, a simple example of functions satisfying \eqref{convexE} is given by
		\begin{equation*}
			E_i(y)=\frac{a_i}{(1+y)^{b_i}},\quad\text{ with }a_i>0\text{ and }b_i\in (0,1).
		\end{equation*}
		In this case indeed it holds
		\begin{equation*}
			\frac 12 E_i''(y)E_i(y)-E_i'(y)^2=\frac{a_i^2}{2}\frac{b_i(1-b_i)}{(1+y)^{2(1+b_i)}}\ge \frac{a_i^2}{2}\frac{b_i(1-b_i)}{4^{1+b_i}}=m_i.
		\end{equation*}
		Moreover $M_i=\max\limits_{y\in[0,1]}E_i''(y)= a_i b_i(1+b_i)$, so that
		\begin{equation*}
			\frac{m_i}{M_i}=\frac{a_i}{2}\frac{1-b_i}{1+b_i}\frac{1}{4^{1+b_i}}.
		\end{equation*}
		In the particular case in which $a_1=a_2=:a$ and $b_1=b_2=1/2$ we get $\frac{m_1}{M_1}=\frac{m_2}{M_2}=\frac{a}{48}$ and so condition \eqref{maincondition} can be written as:
		\begin{equation*}
			\frac{a}{\lambda L^2}>\frac{48}{\pi^2},
		\end{equation*}
		and can be achieved by increasing the parameter $a$ or by decreasing $\lambda$ or the lenght of the bar $L$.
	\end{rmk}
	For convenience, in this section we also introduce the notation of the ``shifted'' energy, see also \cite{MielkRoub}, Remark 3.2. For $t\in [0,T]$ and $x\in [0,L]$ we define the 
	function $\bar{\bm{u}}_D(t,x):=\left(\frac{\bar{u}(t)}{L}x,\frac{\bar{u}(t)}{L}x\right)$ and we present the shifted variable $\bm{v}(t)=\bm{u}(t)-\bar{\bm{u}}_D(t)$, which has 
	zero boundary conditions and hence it belongs to $ [H^1_0(0,L)]^2$. We finally introduce the shifted energy:
	\begin{equation*}
		\mc E_D[t,\bm{v}(t),\bm{\alpha}(t)]:=\mc E[\bm{v}(t)+\bar{\bm{u}}_D(t),\bm{\alpha}(t)]=\mc E[\bm{u}(t),\bm{\alpha}(t)],
	\end{equation*}
	and we want to highlight its explicit dependence on time given by the prescribed displacement and encoded by the function $\bar{\bm{u}}_D$.
	Written in this form, the energy allows us to recast the work of the external prescribed displacement \eqref{work} in the following way:
	\begin{equation}\label{power}
		\mc W[\bm{u},\bm{\alpha}](t)=\int_{0}^{t}\partial_t\mc E_D[\tau,\bm{v}(\tau),\bm{\alpha}(\tau)]\d\tau.
	\end{equation}
	Moreover, by simple computations, it is easy to see that for almost every time $\tau\in [0,T]$ and for every $t\in [0,T]$ the following inequality holds true:
	\begin{equation}\label{powercontrol}
		\begin{aligned}
			&\quad|\partial_t\mc E_D[\tau,\bm{v}(\tau),\bm{\alpha}(\tau)]-\partial_t\mc E_D[\tau,\bm{v}(t),\bm{\alpha}(t)]|\\
			&\le C|\dot{\bar{u}}(\tau)|\Big(\Vert\bm{\alpha}(\tau)-\bm{\alpha}(t)\Vert^2_{[H^1(0,L)]^2}+\Vert\bm{v}(\tau)'-\bm{v}(t)'\Vert^2_{[L^2(0,L)]^2}\Big)^\frac 12,
		\end{aligned}
	\end{equation}
	where $C>0$ is a suitable positive constant.\par
	Furthermore we also notice that the global stability condition (GS') of Definition~\ref{Generalisedenersol} can be rewritten as:
	\begin{itemize}
		\item[] for every $t\in[0,T]$ one has $\gamma(t)\ge \delta(t)$ in $[0,L]$ and
		\begin{equation*}
			\mathcal{E}_D[t,\bm{v}(t),\bm{\alpha}(t)]+\mathcal D[\bm{\alpha}(t)]+\mathcal{K}[\delta(t),\gamma(t)]\le\mathcal{E}_D[t,{\widetilde{\bm v}},{\widetilde{\bm\alpha}}]
			+\mathcal D[{\widetilde{\bm\alpha}}]+\mathcal{K}[\widetilde\delta,\gamma(t)\vee\widetilde\delta],
		\end{equation*}
		for every ${\widetilde{\bm v}}\in [H^1_0(0,L)]^2$ and for every ${\widetilde{\bm\alpha}}\in [H^1(0,L)]^2$ such that ${\alpha}_i(t)\le{\widetilde{\alpha}_i}\le {1}$ in $[0,L]$ 
		for $i=1,2$;
	\end{itemize}
	We finally have all the ingredients to start the analysis regarding the temporal regularity of generalised energetic evolutions. We first state a useful lemma, whose simple proof 
	can be found for instance in \cite{GidRiv}, Lemma~5.6, or in \cite{HeiMielk}, Lemma~4.3.
	\begin{lemma}\label{inequalityintegral}
		Let $(X,\Vert\cdot\Vert)$ be a normed space and let $f\colon[a,b]\to X$ be a bounded measurable function such that:
		\begin{equation*}
			\Vert f(t)-f(s)\Vert^2\le \int_{s}^{t}\Vert f(t)-f(\tau)\Vert g(\tau)d\tau,\quad\text{for every }a\le s\le t\le b,
		\end{equation*}
		for some nonnegative $g\in L^1(a,b)$. Then actually it holds:
		\begin{equation*}
			\Vert f(t)-f(s)\Vert\le \int_{s}^{t}g(\tau)d\tau,\quad\text{for every }a\le s\le t\le b.
		\end{equation*}
	\end{lemma}
	We are now in a position to state and prove the first  result of this section, which yields temporal regularity of generalised energetic evolutions under the convexity assumptions 
	we stated before. The argument is based on the ideas of \cite{Thom}, adapted to our setting where also a cohesive energy (concave by nature) is taken into account.
	\begin{prop}[\textbf{Temporal Regularity}]\label{propregularity}
		Assume $E_i$ satisfies \eqref{convexE}, $w_i$ satisfies \eqref{convexw}, and assume $\varphi\in C^0(\mc T)$ satisfies ($\varphi$3), ($\varphi$5)--($\varphi$8). 
		Let $(\bm{u},\bm{\alpha},\gamma)$ be a generalised energetic evolution related to the prescribed displacement $\bar{u}\in AC([0,T])$. If condition \eqref{maincondition} 
		on the parameters is satisfied, then both the displacements $\bm{u}$ and the damage variables $\bm{\alpha}$ belong to $ AC([0,T];[H^1(0,L)]^2)$, and so one also has 
		$\delta\in AC([0,T];H^1(0,L))$ and $\delta_h\in AC([0,T]; C^0([0,L]))$.\par
		If in addition the family $\{\gamma(t)\wedge\bar\delta\}_{t\in[0,T]}$, with $\bar\delta$ introduced in \eqref{bardelta}, is equicontinuous and for every $y\in[0,\bar\delta)$ the map $\varphi(y,\cdot)$ is strictly 
		increasing in $[y,\bar\delta)$, then the function $\gamma\wedge \bar\delta$ belongs to $ C^0([0,T];C^0([0,L])$.
	\end{prop}
	\begin{rmk}\label{notrestrictive}
		We want to point out that the additional requirement of equicontinuity of the family $\{\gamma(t)\wedge\bar\delta\}_{t\in[0,T]}$, although can not be derived directly from the 
		Definition~\ref{Generalisedenersol} of generalised energetic evolutions, is automatically satisfied by the limit function $\gamma$ obtained in Proposition~\ref{limits}. Thus it is not restrictive.
	\end{rmk}
	\begin{proof}[Proof of Proposition~\ref{propregularity}]
		We first consider, for $i=1,2$, the Hessian matrix of the function $[0,1]\times\erre\ni(\alpha,v)\mapsto \frac 12 E_i(\alpha)v^2$, denoted by $H_i(\alpha,v)$, and its quadratic form, 
		namely the map:
		\begin{equation*}
			(x,y)\mapsto\langle(x,y),H_i(\alpha,v)(x,y)\rangle=\frac 12 E_i''(\alpha)v^2x^2+2 E_i'(\alpha)vxy+E_i(\alpha)y^2.
		\end{equation*}
		By \eqref{convexE} it must be $E_i''(\alpha)>0$ for every $\alpha\in [0,1]$ and so we can write:
		\begin{align*}
			\langle(x,y),H_i(\alpha,v)(x,y)\rangle&=\frac{2}{E_i''(\alpha)}\left[\left(\frac 12 E_i''(\alpha)vx+E_i'(\alpha)y\right)^2
			+\left(\frac 12 E_i''(\alpha)E_i(\alpha)-E_i'(\alpha)^2\right)y^2\right]\\
			&\ge 2\frac{m_i}{M_i}y^2.
		\end{align*}
		Thanks to this estimate on the Hessian matrix it is easy to infer that for every $t\in [0,T]$, for every $\theta\in[0,1]$ and for every $\bm{v}^a,\bm{v}^b\in [H^1_0(0,L)]^2$ 
		and $\bm{\alpha}^a,\bm{\alpha}^b\in [H^1_{[0,1]}(0,L)]^2$ it holds:
		\begin{equation}\label{thetaE}
			\begin{aligned}
				&\qquad \mc E_D[t,\theta\bm{v}^a+(1-\theta)\bm{v}^b,\theta\bm{\alpha}^a+(1-\theta)\bm{\alpha}^b]\\
				&\le \theta \mc E_D[t,\bm{v}^a,\bm{\alpha}^a]+(1-\theta)\mc E_D[t,\bm{v}^b,\bm{\alpha}^b]
				-\frac{m_1}{M_1}\wedge\frac{m_2}{M_2}\theta(1-\theta)\Vert(\bm{v}^a)'-(\bm{v}^b)'\Vert^2_{[L^2(0,L)]^2}.
			\end{aligned}	
		\end{equation}
		By means of \eqref{convexw} we also deduce that for every $t\in [0,T]$, for every $\theta\in[0,1]$ and for every $\bm{\alpha}^a,\bm{\alpha}^b\in [H^1_{[0,1]}(0,L)]^2$ we have:
		\begin{equation}\label{thetaD}
			\mc D[\theta\bm{\alpha}^a+(1-\theta)\bm{\alpha}^b]\le \theta \mc D[\bm{\alpha}^a]+(1-\theta)\mc D[\bm{\alpha}^b]
			-\frac{\mu\wedge 1}{2}\theta(1-\theta)\Vert\bm{\alpha}^a-\bm{\alpha}^b\Vert^2_{[H^1(0,L)]^2}.
		\end{equation}
		Finally, by \eqref{2c}, \eqref{convexpsi} and \eqref{convexphi} (which are implied by ($\varphi$5)--($\varphi$8)) we deduce that for every $z\in[0,+\infty)$ the function 
		$y\mapsto\varphi(y,z\vee y)$ is $\lambda$--convex in $[0,+\infty)$; thus for every $t\in [0,T]$, for every $\theta\in[0,1]$ and for every nonnegative 
		$\delta^a,\delta^b\in H^1_0(0,L)$ it holds:
		\begin{equation}\label{thetaK}
			\begin{aligned}
				&\qquad\mc K[\theta\delta^a+(1-\theta)\delta^b,\gamma(t)\vee(\theta\delta^a+(1-\theta)\delta^b)]\\
				&\le \theta\mc K[\delta^a,\gamma(t)\vee\delta^a]+(1-\theta)\mc K[\delta^b,\gamma(t)\vee\delta^b]+\frac{\lambda}{2}\theta(1-\theta)\Vert\delta^a-\delta^b\Vert^2_{L^2(0,L)}.
			\end{aligned}	
		\end{equation}
		We now fix $t\in[0,T]$, $\theta\in (0,1)$, $\widetilde{\bm{v}}\in[H^1_0(0,L)]^2$, $\widetilde{\bm{\alpha}}\in[H^1(0,L)]^2$ such that $\alpha_i(t)\le\widetilde{{\alpha}}_i(t)
		\le 1$ for $i=1,2$, and we consider as competitors in (GS') the functions $\theta \widetilde{\bm{v}}+(1-\theta)\bm{v}(t)$ and 
		$\theta \widetilde{\bm{\alpha}}+(1-\theta)\bm{\alpha}(t)$; by means of \eqref{thetaE}, \eqref{thetaD} and \eqref{thetaK}, together with ($\varphi$3) and \eqref{convexphi}, 
		we thus get:
		\begin{align}\label{esteta}
			&\qquad\mc E_D[t,\bm{v}(t),\bm{\alpha}(t)]+\mc D[\bm{\alpha}(t)]+\mc K[\delta(t),\gamma(t)]\nonumber\\
			&\le\mc E_D[t,\theta \widetilde{\bm{v}}+(1-\theta)\bm{v}(t),\theta \widetilde{\bm{\alpha}}+(1-\theta)\bm{\alpha}(t)]
			+\mc D[\theta \widetilde{\bm{\alpha}}+(1-\theta)\bm{\alpha}(t)]\nonumber\\
			&\quad+\mc K[|\theta(\widetilde{v}_1-\widetilde{v}_2)+(1-\theta)(v_1(t)-v_2(t))|,\gamma(t)\vee|\theta(\widetilde{v}_1-\widetilde{v}_2)+(1-\theta)(v_1(t)-v_2(t))|]\nonumber\\
			&\le \theta \mc E_D[t,\widetilde{\bm{v}},\widetilde{\bm{\alpha}}]+(1-\theta)\mc E_D[t,\bm{v}(t),\bm{\alpha}(t)]
			-\frac{m_1}{M_1}\wedge\frac{m_2}{M_2}\theta(1-\theta)\Vert(\widetilde{\bm{v}})'-(\bm{v}(t))'\Vert^2_{[L^2(0,L)]^2}\nonumber \\
			&\quad+\theta \mc D[\widetilde{\bm{\alpha}}]+(1-\theta)\mc D[\bm{\alpha}(t)]-\frac{\mu\wedge 1}{2}\theta(1-\theta)\Vert\widetilde{\bm{\alpha}}
			-\bm{\alpha}(t)\Vert^2_{[H^1(0,L)]^2}\nonumber \\
			&\quad +\theta\mc K[\widetilde\delta,\gamma(t)\vee\widetilde\delta]
			+(1-\theta)\mc K[\delta(t),\gamma(t)]+\frac{\lambda}{2}\theta(1-\theta)\Vert\widetilde\delta-\delta(t)\Vert^2_{L^2(0,L)}.
		\end{align}
		We now exploit the well known sharp Poincar\'{e} inequality:
		\begin{equation*}
			\int_{a}^{b}f(x)^2\d x\le\frac{(b-a)^2}{\pi^2}\int_{a}^{b} f'(x)^2\d x,\text{ for every }f\in H^1_0(a,b),
		\end{equation*}
		to deduce that
		\begin{equation}\label{deltav}
			\Vert\widetilde\delta-\delta(t)\Vert^2_{L^2(0,L)}\le 2\frac{L^2}{\pi^2}\Vert(\widetilde{\bm{v}})'-(\bm{v}(t))'\Vert^2_{[L^2(0,L)]^2}.
		\end{equation}
		By plugging \eqref{deltav} in \eqref{esteta}, dividing by $\theta$ and then letting $\theta\to 0^+$ we finally deduce:
		\begin{equation}\label{improvedstability}
			\begin{aligned}
				&\quad\left(\frac{m_1}{M_1}\wedge\frac{m_2}{M_2}-\lambda\frac{L^2}{\pi^2} \right)\Vert(\widetilde{\bm{v}})'-(\bm{v}(t))'\Vert^2_{[L^2(0,L)]^2}
				+\frac{\mu\wedge 1}{2}\Vert\widetilde{\bm{\alpha}}-\bm{\alpha}(t)\Vert^2_{[H^1(0,L)]^2}\\
				&\quad+\mc E_D[t,\bm{v}(t),\bm{\alpha}(t)]+\mc D[\bm{\alpha}(t)]+\mc K[\delta(t),\gamma(t)]\\
				&\le\mc E_D[t,\widetilde{\bm{v}},\widetilde{\bm{\alpha}}]+\mc D[\widetilde{\bm{\alpha}}]+\mc K[\widetilde\delta,\gamma(t)\vee\widetilde\delta].
			\end{aligned}
		\end{equation}
		For the sake of simplicity we denote by $c$ the minimum between $\frac{m_1}{M_1}\wedge\frac{m_2}{M_2}-\lambda\frac{L^2}{\pi^2}$ and $\frac{\mu\wedge 1}{2}$, 
		and we notice that $c$ is strictly positive by \eqref{maincondition}. We now fix two times $0\le s\le t\le T$. Exploiting \eqref{improvedstability} at time $s$ with 
		$\widetilde{\bm{v}}=\bm{v}(t)$ and $\widetilde{\bm{\alpha}}=\bm{\alpha}(t)$, and recalling (EB') and \eqref{power} we obtain:
		\begin{align*}
			&\quad c\left(\Vert\bm{\alpha}(t)-\bm{\alpha}(s)\Vert^2_{[H^1(0,L)]^2}+\Vert\bm{v}(t)'-\bm{v}(s)'\Vert^2_{[L^2(0,L)]^2}\right)\\
			&\le \mc E_D[s,\bm{v}(t),\bm{\alpha}(t)]+\mc D[\bm{\alpha}(t)]+\mc K[\delta(t),\gamma(s)\vee\delta(t)]-\mc E_D[s,\bm{v}(s),\bm{\alpha}(s)]
			-\mc D[\bm{\alpha}(s)]-\mc K[\delta(s),\gamma(s)]\\
			&\le \mc E_D[s,\bm{v}(t),\bm{\alpha}(t)]-\mc E_D[t,\bm{v}(t),\bm{\alpha}(t)]+\mc W[\bm{u},\bm{\alpha}](t)-\mc W[\bm{u},\bm{\alpha}](s)\\
			&\le \int_{s}^{t}|\partial_t\mc E_D[\tau,\bm{v}(\tau),\bm{\alpha}(\tau)]-\partial_t\mc E_D[\tau,\bm{v}(t),\bm{\alpha}(t)]|\d\tau.
		\end{align*}
		By using \eqref{powercontrol} we thus deduce:
		\begin{align*}
			&\quad\Vert\bm{\alpha}(t)-\bm{\alpha}(s)\Vert^2_{[H^1(0,L)]^2}+\Vert\bm{v}(t)'-\bm{v}(s)'\Vert^2_{[L^2(0,L)]^2}\\
			&\le \frac Cc \int_{s}^{t}|\dot{\bar{u}}(\tau)|\Big(\Vert\bm{\alpha}(\tau)-\bm{\alpha}(t)\Vert^2_{[H^1(0,L)]^2}
			+\Vert\bm{v}(\tau)'-\bm{v}(t)'\Vert^2_{[L^2(0,L)]^2}\Big)^\frac 12\d\tau.
		\end{align*}
		By means of \eqref{bounds1sol} we can apply Lemma~\ref{inequalityintegral} getting:
		\begin{equation*}
			\Big(\Vert\bm{\alpha}(t)-\bm{\alpha}(s)\Vert^2_{[H^1(0,L)]^2}+\Vert\bm{v}(t)'-\bm{v}(s)'\Vert^2_{[L^2(0,L)]^2}\Big)^\frac 12
			\le \frac Cc\int_{s}^{t}|\dot{\bar{u}}(\tau)|\d\tau,
		\end{equation*}
		and so we infer that $\bm{\alpha}$ belongs to $AC([0,T];[H^1(0,L)]^2)$ and $\bm{v}$ belongs to $AC([0,T];[H^1_0(0,L)]^2)$. By construction we also have
		\begin{align*}
			\Vert\bm{u}(t)-\bm{u}(s)\Vert_{[H^1(0,L)]^2}&\le \Vert\bm{v}(t)-\bm{v}(s)\Vert_{[H^1(0,L)]^2}+\Vert\bm{u}_D(t)-\bm{u}_D(s)\Vert_{[H^1(0,L)]^2}\\
			&\le  \Vert\bm{v}(t)-\bm{v}(s)\Vert_{[H^1(0,L)]^2}+C|\bar{u}(t)-\bar{u}(s)|,
		\end{align*}
		so also $\bm{u}$ belongs to $AC([0,T];[H^1(0,L)]^2)$ and as a simple byproduct we obtain that $\delta$ is $AC([0,T];H^1(0,L))$.\par
		Since $H^1(0,L)\subseteq C^0([0,L])$, in particular there exists a nonnegative function $\phi\in L^1(0,T)$ such that
		\begin{equation}\label{daprovare}
			\Vert\delta(t)-\delta(s)\Vert_{C^0([0,L])}\le \int_{s}^{t}\phi(\tau)\d\tau,\text{ for every }0\le s\le t\le T.
		\end{equation}
		We now show that the same inequality holds true for $\delta_h$ in place of $\delta$. We thus fix $0\le s\le t\le T$ and $x\in [0,L]$. 
		If $\delta_h(t,x)=\delta_h(s,x)$ there is nothing to prove, so let us assume $\delta_h(t,x)>\delta_h(s,x)$. By definition of $\delta_h$ and since 
		now we know that $\delta$ is continuous both in time and space we deduce that
		$$\delta_h(t,x)=\max\limits_{\tau\in[0,t]}\delta(\tau,x)=\delta(t_x,x),\quad\text{ for some }t_x\in[s,t].$$
		So we have
		\begin{equation*}
			\delta_h(t,x)-\delta_h(s,x)\le \delta(t_x,x)-\delta(s,x)\le\int_{s}^{t_x}\phi(\tau)\d\tau\le \int_{s}^{t}\phi(\tau)\d\tau.
		\end{equation*}
		We have thus proved the validity of \eqref{daprovare} with $\delta_h$ in place of $\delta$, and hence $\delta_h$ belongs to $AC([0,T];C^0([0,L]))$.\par
		We only need to prove that $\gamma\wedge\bar\delta\in C^0([0,T];C^0([0,L])$ under the additional assumptions that $\{\gamma(t)\wedge\bar\delta\}_{t\in[0,T]}$ 
		is an equicontinuous family and $\varphi(y,\cdot)$ is strictly increasing in $[y,\bar\delta)$ for any given $y\in [0,\bar\delta)$. For the sake of clarity we prove 
		it only in the case $\bar\delta=+\infty$; in the other situation the result can be obtained arguing in the same way and recalling equality \eqref{equalitiphidelta}. 
		To this aim we observe that, by equicontinuity, for every $t\in [0,T]$ the right and the left limits $\gamma^+(t)$ and $\gamma^-(t)$ are continuous in $[0,L]$. 
		By monotonicity and using classical Dini's theorem we hence obtain
		\begin{equation}\label{uniform}
			\gamma^\pm(t)=\lim\limits_{h\to 0^\pm}\gamma(t+h),\quad\text{ uniformly in  }[0,L].
		\end{equation}
		So we conclude if we prove that $\gamma^+(t)=\gamma^-(t)$.\par
		Arguing as in the proof of Proposition~\ref{propimportant}, since $\bm{u}$ and $\bm{\alpha}$ are in $AC([0,T];[H^1(0,L)]^2)$, we deduce by (EB') that the map 
		$t\mapsto \mc K[\delta(t),\gamma(t)]$ is continuous in $[0,T]$, and thus for every $t\in [0,T]$ we have:
		\begin{equation*}
			\lim\limits_{h\to 0^+}\mc K[\delta(t+h),\gamma(t+h)]=\lim\limits_{h\to 0^-}\mc K[\delta(t+h),\gamma(t+h)].
		\end{equation*}
		By using \eqref{uniform} we can pass to the limit inside the integral getting
		\begin{equation*}
			\int_{0}^{L}\varphi(\delta(t),\gamma^+(t))\d x=\int_{0}^{L}\varphi(\delta(t),\gamma^-(t))\d x.
		\end{equation*}
		Since $\varphi(y,\cdot)$ is strictly increasing we conclude.
	\end{proof}
	Thanks to the temporal regularity obtained in the previous proposition we are able to prove our main results. The first theorem ensures the equality between 
	$\gamma$ and $\delta_h$ (actually between $\gamma\wedge\bar\delta$ and $\delta_h\wedge\bar\delta$, which however are the meaningful ones, see Remark~\ref{Rmkimportant}) 
	assuming a priori equicontinuity on the family $\{\gamma(t)\}_{t\in[0,T]}$, which is however not restrictive due to Remark~\ref{notrestrictive}; 
	a similar argument to the one adopted here, but in an easier setting, can be found in \cite{Rivquas}, Proposition~2.7. 
	The second theorem states that the generalised energetic evolution obtained in Section~\ref{secexistence} as limit of discrete minimisers is actually an energetic evolution. 
	We thus reach our goal, avoiding the assumption ($\varphi$4), and considering the list of reasonable assumptions ($\varphi$5)--($\varphi$9) 
	(actually we replace ($\varphi$9) by the weaker \eqref{bilipass}) which for instance are satisfied by the example provided in Remark~\ref{rmkexamplephi}.
	\begin{thm}[\textbf{Equivalence between $\gamma$ and $\delta_h$}]\label{thmgammadeltah}
		Let the prescribed displacement $\bar{u}$ belong to the space $ AC([0,T])$. Assume $E_i$ satisfies \eqref{convexE}, $w_i$ satisfies \eqref{convexw}, 
		and $\varphi\in C^0(\mc T)$ satisfies ($\varphi$5)--($\varphi$8), plus the following uniform strict monotonicity with respect to $z$:
		\begin{equation}\label{bilipass}
			\begin{gathered}
				\text{ for every compact set }K\in \{z>y\ge 0\}\cap\overline{\mc T_{\bar\delta}}\text{ there exists a positive constant }C_K>0 \text{ such that}\\
				\varphi(y,z_2)-\varphi(y,z_1)\ge C_K(z_2-z_1) \\
				\text{ for every }(z_2,y),(z_1,y)\in K\text{ satisfying }z_2\ge z_1.
			\end{gathered}
		\end{equation}
		Assume also condition \eqref{maincondition} on the parameters. Then, given a generalised energetic evolution $(\bm{u},\bm{\alpha},\gamma)$ such that the family 
		$\{\gamma(t)\wedge\bar\delta\}_{t\in[0,T]}$ is equicontinuous, the function $\gamma\wedge\bar\delta$ coincides with $\delta_h\wedge\bar\delta$.
	\end{thm}
	\begin{proof}
		For the sake of clarity we prove the result only in the case $\bar\delta=+\infty$, being the other situation analogous by \eqref{equalitiphidelta}.\par
		We know that $\gamma\ge \delta_h$ and that $\gamma(0)=\delta_h(0)=\delta^0$ and $\gamma(t,0)=\gamma(t,L)=\delta_h(t,0)=\delta_h(t,L)=0$ for every $t\in [0,T]$. 
		Moreover by Proposition~\ref{propregularity} we know that both $\gamma$ and $\delta_h$ are continuous on $[0,T]\times[0,L]$.\par
		We thus assume by contradiction there exists $(\bar t,\bar x)\in (0,T]\times(0,L)$ for which $\gamma(\bar t,\bar x)>\delta_h(\bar t,\bar x)$; by continuity we thus 
		deduce there exists $\eta>0$ such that
		\begin{equation*}
			\gamma(t,x)>\delta_h(t,x)\ge\delta(t,x),\quad\text{ for every }(t,x)\in [\bar t-\eta,\bar t\,]\times [\bar x-\eta,\bar x+\eta].
		\end{equation*}
		By assumption \eqref{bilipass} we hence infer the existence of  constant $c_\eta>0$ for which
		\begin{equation}\label{bilipestimate}
			\begin{gathered}
				\varphi(\delta(s,x),\gamma(t,x)){-}\varphi(\delta(s,x),\gamma(s,x))\ge c_\eta(\gamma(t,x){-}\gamma(s,x)),\\
				\text{ for every }\bar t{-}\eta\le s\le t\le \bar t\text{ and }x\!\in\![\bar x{-}\eta,\bar x{+}\eta].
			\end{gathered}
		\end{equation}
		We now recall that by Proposition~\ref{propregularity} we know the map $t\mapsto\mc K[\delta(t),\gamma(t)]$ is absolutely continuous in $[0,T]$. 
		So for every $0\le s\le t\le T$ we can estimate:
		\begin{align}\label{estimate2}
			&\quad\,\,\int_{0}^{L}(\varphi(\delta(s),\gamma(t)){-}\varphi(\delta(s),\gamma(s)))\d x\nonumber\\
			&=\mc K[\delta(t),\gamma(t)]-\mc K[\delta(s),\gamma(s)]
			+\int_{0}^{L}(\varphi(\delta(s),\gamma(t)){-}\varphi(\delta(t),\gamma(t)))\d x\\
			&\le \int_{s}^{t}\frac{\d}{\d t}\mc K[\delta(\tau),\gamma(\tau)]\d\tau+C\Vert\delta(t)-\delta(s)\Vert_{C^0([0,L])}\le \int_{s}^{t}\phi(\tau)\d\tau,\nonumber
		\end{align}
		where $\phi\in L^1(0,T)$ is a suitable nonnegative function.\par
		Combining \eqref{bilipestimate} and \eqref{estimate2} we now obtain:
		\begin{equation*}
			c_\eta\int_{\bar x-\eta}^{\bar x+\eta}(\gamma(t)-\gamma(s))\d x\le \int_{s}^{t}\phi(\tau)\d\tau,\quad\text{ for every }\bar t{-}\eta\le s\le t\le \bar t,
		\end{equation*}
		hence $\gamma\in AC([\bar t-\eta,\bar t\,];L^1(\bar x-\eta,\bar x+\eta))$.\par
		By means of \eqref{zeroder} we now deduce that for a.e. $t\in [\bar t-\eta,\bar t\,]$ we have:
		\begin{equation*}
			0=\lim\limits_{h\to 0}\int_{0}^{L}\frac{\varphi(\delta(t),\gamma(t+h))-\varphi(\delta(t),\gamma(t))}{h}\d x
			\ge c_\eta\limsup\limits_{h\to 0}\int_{\bar x-\eta}^{\bar x+\eta}\frac{\gamma(t+h)-\gamma(t)}{h}\d x\ge 0,
		\end{equation*}
		namely for almost every $t\in[\bar t-\eta,\bar t\,]$ the function $\gamma$ is strongly differentiable in $L^1(\bar x-\eta,\bar x+\eta)$ and $\dot{\gamma}(t)=0$. 
		By Proposition~\ref{sobolevequivalence} we now obtain
		\begin{equation*}
			\gamma(t)=\gamma(\bar t-\eta)+\int_{\bar t-\eta}^{t}\dot{\gamma}(\tau)\d\tau
			=\gamma(\bar t-\eta),\text{ for every }t\in [\bar t-\eta,\bar t\,],\text{ as an equality in }L^1(\bar x-\eta,\bar x+\eta).
		\end{equation*}
		In particular, since $\gamma$ is continuous, we deduce that $\gamma(\bar t,\bar x)=\gamma(\bar t-\eta, \bar x)$.\par
		Since $\delta_h$ is non-decreasing we can iterate all the previous argument, finally getting $\gamma(\bar t,\bar x)=\gamma(0, \bar x)$. But this is a contradiction, indeed it implies:
		\begin{equation*}
			\delta^0(\bar x)=\gamma(0,\bar x)=\gamma(\bar t,\bar x)>\delta_h(\bar t,\bar x)\ge \delta_h(0,\bar x)=\delta^0(\bar x),
		\end{equation*}
		and so we conclude.
	\end{proof}
	
	\begin{thm}[\textbf{Existence of Energetic Evolutions}]\label{finalthm}
		Let the prescribed displacement $\bar{u}$ belong to the space $ AC([0,T])$ and the initial data $\bm{u}^0$, $\bm{\alpha}^0$ fulfil \eqref{compatibility} together with 
		the stability condition \eqref{GS0}. Assume $E_i$ satisfies \eqref{convexE}, $w_i$ satisfies \eqref{convexw}, and $\varphi\in C^0(\mc T)$ 
		satisfies ($\varphi$2), ($\varphi$5)--($\varphi$8) and \eqref{bilipass}. Assume also condition \eqref{maincondition} on the parameters. 
		Then the pair $(\bm{u}, \bm{\alpha})$ composed by the functions obtained in Proposition~\ref{limits} is an energetic evolution, since it 
		holds $\gamma\wedge\bar\delta=\delta_h\wedge\bar\delta$, with $\bar\delta$ introduced in \eqref{bardelta}. \par
		Moreover $\bm{u}$ and $\bm{\alpha}$ belong to $ AC([0,T];[H^1(0,L)]^2)$, and so in particular the history slip $\delta_h$ is in $AC([0,T]; C^0([0,L]))$.
	\end{thm}
	\begin{proof}
		The result is a simple byproduct of Theorem~\ref{exgenensol} together with Proposition~\ref{propregularity} and Theorem~\ref{thmgammadeltah} 
		(we also recall \eqref{equalitiphidelta}). We indeed notice that the equicontinuity assumption on the family $\{\gamma(t)\wedge\bar\delta\}_{t\in[0,T]}$ 
		(actually on the whole $\{\gamma(t)\}_{t\in[0,T]}$) is automatically satisfied by the limit function $\gamma$ obtained in Proposition~\ref{limits}.
	\end{proof}
	
	\bigskip

	\subsection*{Conclusions}
	The obtained results offer new insights for further investigations. The 2D numerical investigations presented in \cite{AleFredd2d}, where the complex failure modes of hybrid laminates are consistently reproduced, suggest to extend the theoretical investigation to higher dimensional settings whereby the introduction of the anisotropic behavior of materials allows the analysis of problems of interest for the conservation of cultural heritage \cite{Bucklow,Negri} and other micro-cracking phenomena such \cite{Qin}. A second line of exploration could also be the analysis of the problem in case of complete damage, meant as complete loss of material stiffness.\par 
	Moreover, it would be interesting to extend the proposed approach to classical problems of cohesive fracture mechanics. In this case, dissipation combined with irreversible effects introduces difficulties, at least when dealing with global minimisers of the energy, in considering loading-unloading cohesive laws that reflect the real behaviour of materials rather than hypotheses dictated by mere mathematical assumptions. The main difference provided by cohesive fracture models with respect to the considered problem of cohesive interface relies in the reduced dimension of the fracture, which is a $(d-1)$--dimensional object in a $d$--dimensional material. This feature involves the use of weaker topologies, which can not be directly treated following our argument, and thus requires further adaptations in order to transfer our results.

	\appendix
	\section{Absolutely Continuous and BV--Vector Valued Functions }
	
	In this Appendix we briefly present the main definitions and properties of vector valued absolutely continuous functions and functions of bounded variation we used 
	throughout the paper. A deeper and more detailed analysis can be found in the Appendix of \cite{Brez}, to which we refer for all the proofs and examples. Here $(X,\Vert\cdot\Vert)$ will denote a Banach space, 
	and by $X^*$ we mean its topological dual. The duality product between $w\in X^*$ and $x\in X$ is finally denoted by $\langle w,x\rangle$.
	\begin{defi}
		A function $f\colon [0,T]\to X$ is said to be:
		\begin{itemize}
			\item a function of bounded variation ($BV([0,T];X)$) if $$V_X(f;0,T):=\sup\limits_{\substack{\text{finite partitions}\\
					\text{of }[0,T]}}\sum\Vert f(t_k)-f(t_{k-1}) \Vert<+\infty;$$
			\item absolutely continuous ($AC([0,T];X)$) if there exists a nonnegative function $\phi\in L^1(0,T)$ such that
			$$\Vert f(t)-f(s)\Vert\le \int_{s}^{t}\phi(\tau)\d\tau,\quad\text{ for every }0\le s\le t\le T;$$
			\item in the space $\widetilde{W}^{1,p}(0,T;X)$, $p\in[1,+\infty]$, if there exists a nonnegative function $\phi\in L^p(0,T)$ such that
			$$\Vert f(t)-f(s)\Vert\le \int_{s}^{t}\phi(\tau)\d\tau,\quad\text{ for every }0\le s\le t\le T;$$
			\item in the Sobolev space $W^{1,p}(0,T;X)$, $p\in[1,+\infty]$, if there exists a function $g\in L^p(0,T;X)$ such that
			$$f(t)=f(0)+\int_{0}^{t}g(\tau)\d\tau,\quad\text{ for every }t\in [0,T].$$
		\end{itemize}
	\end{defi}
	As in the classical case $X=\erre$ any function of bounded variation belongs to $L^\infty(0,T;X)$, it admits right and left (strong) limits at every $t\in[0,T]$ and the 
	set of its discontinuity points is at most countable. To gain the well known property of almost everywhere differentiability also in the vector valued framework it is instead 
	crucial to require $X$ to be reflexive (see the examples in \cite{Brez}).
	\begin{prop}\label{reflexive}
		If $X$ is reflexive, then any function $f$ belonging to $BV([0,T];X)$ is weakly differentiable almost everywhere in $[0,T]$. 
		Moreover $\Vert\dot{f}(t)\Vert\le \frac{\d}{\d t}V_X(f;0,t)$ for a.e. $t\in[0,T]$ and in particular $\dot{f}\in L^1(0,T;X)$.
	\end{prop}
	We now focus our attention on absolutely continuous and Sobolev functions. By the very definition it is easy to see that any absolutely continuous function is also of 
	bounded variation; furthermore the spaces $AC([0,T];X)$ and $\widetilde{W}^{1,1}(0,T;X)$ coincide, while $\widetilde{W}^{1,\infty}(0,T;X)$ is the space of Lipschitz 
	functions from $[0,T]$ to $X$. Moreover for every $p\in[1,+\infty]$ the inclusion $W^{1,p}(0,T;X)\subseteq \widetilde W^{1,p}(0,T;X)$ always holds, but in general is strict. \par
	The next proposition states that the Sobolev space $W^{1,p}(0,T;X)$ is actually characterised by the strong differentiability of its elements.
	\begin{prop}\label{sobolevequivalence}
		Let $p\in[1,+\infty]$ and let $f$ be a function from $[0,T]$ to $X$. Then the following are equivalent:
		\begin{itemize}
			\item[(i)] $f\in  W^{1,p}(0,T;X)$;
			\item[(ii)] $f\in  \widetilde W^{1,p}(0,T;X)$ and it is strongly differentiable for a.e. $t\in[0,T]$;
			\item[(iii)] for every $w\in X^*$ the map $t\mapsto \langle w, f(t)\rangle$ is absolutely continuous in $[0,T]$, $f$ is weakly differentiable for a.e. 
			$t\in[0,T]$ and $\dot{f}\in L^p(0,T;X)$.
		\end{itemize}
		If one of the above condition holds, then one has
		\begin{equation}\label{sob}
			f(t)=f(0)+\int_{0}^{t}\dot{f}(\tau)\d\tau,\quad\text{ for every }t\in [0,T].
		\end{equation}
	\end{prop}
	In the reflexive case, as in Proposition~\ref{reflexive}, we gain differentiability of absolutely continuous functions and so we deduce the equivalence between the 
	two spaces $\widetilde W^{1,p}(0,T;X)$ and $W^{1,p}(0,T;X)$.
	\begin{prop}
		If $X$ is reflexive, then for every $p\in[1,+\infty]$ the Sobolev space $W^{1,p}(0,T;X)$ coincides with $\widetilde W^{1,p}(0,T;X)$.
	\end{prop}
	
	\bigskip
	
	\noindent\textbf{Acknowledgements.}
	E. Bonetti, C. Cavaterra and F. Riva are members of the Gruppo Nazionale per l'Analisi Matematica, la Probabilità e le loro Applicazioni (GNAMPA) of the Istituto Nazionale di Alta Matematica
	(INdAM).\par
	F. Riva acknowledges the support of SISSA (via Bonomea 265, Trieste, Italy), where he was affiliated when this research was carried out.

	{\small
		
		\vspace{15pt} (Elena Bonetti) Universit\`{a} degli Studi di Milano, Dipartimento di Matematica ``Federigo Enriques'',\par\textsc{Via Cesare Saldini, 50, 20133, Milano, Italy}
		\par
		\textit{e-mail address}: \textsf{elena.bonetti@unimi.it}\par
		\textit{Orcid}: \textsf{
			https://orcid.org/0000-0002-8035-3257}
		\par
		
		\vspace{15pt} (Cecilia Cavaterra) Universit\`{a} degli Studi di Milano, Dipartimento di Matematica ``Federigo Enriques'',\par\textsc{Via Cesare Saldini, 50, 20133, Milano, Italy}
		\par
		Istituto di Matematica Applicata e Tecnologie Informatiche “Enrico Magenes”, CNR,\par\textsc{Via Ferrata, 1, 27100, Pavia, Italy}
		\par
		\textit{e-mail address}: \textsf{cecilia.cavaterra@unimi.it}
		\par
		\textit{Orcid}: \textsf{http://orcid.org/0000-0002-2754-7714}
		\par
		
		\vspace{15pt} (Francesco Freddi) Universit\`{a} degli Studi di Parma, Dipartimento di Ingegneria e Architettura, \par \textsc{Parco Area delle Scienze, 181/A, 43124, Parma, Italy}
		\par
		\textit{e-mail address}: \textsf{francesco.freddi@unipr.it}
		\par
		\textit{Orcid}: \textsf{https://orcid.org/0000-0003-0601-6022}
		\par
		
		\vspace{15pt} (Filippo Riva) Universit\`{a} degli Studi di Pavia, Dipartimento di Matematica ``Felice Casorati'', \par
		\textsc{Via Ferrata, 5, 27100, Pavia, Italy}
		\par
		\textit{e-mail address}: \textsf{filippo.riva@unipv.it}
		\par
		\textit{Orcid}: \textsf{https://orcid.org/0000-0002-7855-1262}
		\par
		
	}
	

\bigskip
	
	\begin{thebibliography}{1}
		{
			\bibitem{AleCrismOrl}
			{\sc R.~Alessi, V.~Crismale and G.~Orlando}, {\em Fatigue effects in elastic materials with variational damage models: a vanishing viscosity approach.}, 
			J. Nonlinear Sci., 29 (2019), pp.~1041--1094.
			
			\bibitem{AleFredd2d}
			{\sc R.~Alessi and F.~Freddi}, {\em Failure and complex crack patterns in hybrid laminates: A phase-field approach}, Composite Part B: Engineering, 179 (2019), 107256.
			
			\bibitem{AleFredd1d}
			{\sc R.~Alessi and F.~Freddi}, {\em Phase--field modelling of failure in hybrid laminates}, Composite Structures, 181 (2017), pp.~9--25.
			
			\bibitem{AmbrTort1}
			{\sc L.~Ambrosio and V.~M.~Tortorelli}, {\em Approximation of functionals depending on jumps by elliptic functionals via $\Gamma$-convergence}, 
			Comm. Pure Appl. Math., 43 (1990), pp.~999--1036.
			
			\bibitem{AmbrTort2}
			{\sc L.~Ambrosio and V.~M.~Tortorelli}, {\em On the approximation of free discontinuity problems}, Boll.
			Un. Mat. Ital. B, 6 (1992), pp.~105--123.
			
			\bibitem{Barenblatt}
			{\sc G.~I.~Barenblatt}, {\em The mathematical theory of equilibrium cracks in brittle fracture},  Advances
			in Applied Mechanics, 7 (1962), pp.~55--129.
			
			\bibitem{BonConIurl}
			{\sc M.~Bonacini, S.~Conti and F.~Iurlano}, {\em A phase-field approach to quasi-static evolution for a cohesive fracture model}, Archive for Rational Mechanics and Analysis, 239 (2021), pp.~1501--1576.
			
			\bibitem{BonBonfRos1}
			{\sc E.~Bonetti, G.~Bonfanti and R.~Rossi}, {\em Well-posedness and long-time behaviour for a model of contact with adhesion}, Indiana Univ. Math., 56 (2007), pp.~2787--2819.
			
			\bibitem{BonBonfRos2}
			{\sc E.~Bonetti, G.~Bonfanti and R.~Rossi}, {\em Global existence for a contact problem with adhesion}, Math. Meth. Appl. Sci., 31 (2008), pp.~1029--1064.
			
			\bibitem{BonBonfRos3}
			{\sc E.~Bonetti, G.~Bonfanti and R.~Rossi}, {\em Thermal effects in adhesive contact: modelling and analysi}, Nonlinearity, 22 (2009), pp.~2697--2731.
			
			\bibitem{BonBonfRos4}
			{\sc E.~Bonetti, G.~Bonfanti and R.~Rossi}, {\em Analysis of a unilateral contact problem taking into account adhesion and friction}, Journal of Differential Equations, 253 (2012), pp.~438--462.
			
			\bibitem{BonFreSeg}
			{\sc E.~Bonetti, F.~Freddi and A.~Segatti}, {\em An existence result for a model of complete damage in elastic materials with reversible evolution}, Continuum Mechanics and Thermodynamics, 29 (2017), pp.~31--50.
			
			\bibitem{BonSchim}
			{\sc E.~Bonetti and G.~Schimperna}, {\em Local existence for Fr\'emond’s model of damage in elastic materials}, Continuum Mech. Thermodyn., 16 (2004), pp.~319--335.
			
			\bibitem{BouchMielkRoub}
			{\sc G.~Bouchitt\'e, A.~Mielke and T.~Roub\'i\v cek}, {\em A complete-damage problem at small strains}, Z. Angew. Math. Phys., 60 (2009), pp.~205--236.
			
			\bibitem{BourFrancMar}
			{\sc B.~Bourdin, G.~Francfort and J.--J.~Marigo}, {\em The variational approach to fracture}, Journal of Elasticity, 91 (2008), pp.~5--148.
			
			
			\bibitem{Brez}
			{\sc H.~Brezis}, {\em Operateurs Maximaux Monotones et Semi--groupes de Contractions dans les Espaces de Hilbert},  North--Holland Publishing Company Amsterdam, (1973).
			
			\bibitem{Bucklow}
			{\sc S.~Bucklow}, {\em The description of craquelure patterns}, 
			Stud. Conserv., 42 (2006), pp.~42--129.		
			
			\bibitem{CagnToad}
			{\sc F.~Cagnetti and R.~Toader}, {\em Quasistatic crack evolution for a cohesive zone model  with different response to loading and unloading: a Young measures approach}, 
			ESAIM Control Optim. Calc. Var., 17 (2011), pp.~1--27.
			
			\bibitem{CrisLazzOrl}
			{\sc V.~Crismale, G.~Lazzaroni and G.~Orlando}, {\em Cohesive fracture with irreversibility: quasistatic evolution for a model subject to fatigue}, 
			Math. Models Methods Appl. Sci., 28 (2018), pp.~1371--1412.
			
			\bibitem{DalMasoGamma}
			{\sc G.~Dal Maso}, {\em An Introduction to $\Gamma$-Convergence},  Birkh\"auser, Basel (1993).
			
			\bibitem{DMZan}
			{\sc G.~Dal Maso and C.~Zanini}, {\em Quasistatic crack growth for a cohesive zone model with prescribed crack path}, 
			Proc. Roy. Soc. Edinburgh Sect. A, 137A (2007), pp.~253--279.
			
			\bibitem{FrancMar}
			{\sc G.~Francfort and  J.--J.~Marigo}, {\em Revisiting brittle fracture as an energy minimization problem},  J. Mech. Phys. Solids, 46 (1998), pp.~1319--1342.
			
			\bibitem{FrancMielk}
			{\sc G.~Francfort and  A.~Mielke}, {\em Existence results for a class of rate-independent material models with nonconvex elastic energies},  
			J. Reine Angew. Math., 595 (2006), pp.~55--91.
			
			\bibitem{FreFre}
			{\sc F.~Freddi and  M.~Fr\'emond}, {\em Damage in domains and interfaces: a coupled predictive theory},  J. Mech. Mater. Struct., 7 (2006), pp.~1205--1233.	
			
			
			\bibitem{FreIur}
			{\sc F.~Freddi and F.~Iurlano}, {\em Numerical insight of a variational smeared approach to cohesive fracture},  
			Journal of the Mechanics and Physics of Solids, 98 (2017), pp.~156--171.
			
			\bibitem{Fremond}
			{\sc M.~Fr\'emond, K.~L.~Kuttler and  M.~Shillor}, {\em Existence and uniqueness of solutions for a dynamic one-dimensional damage model},  
			J. of Math. Analysis and Applications, 229 (1999), pp.~271--294.
			
			\bibitem{Giacomini}
			{\sc A.~Giacomini}, {\em Ambrosio-Tortorelli approximation of quasi-static evolution of brittle fractures}, 
			Calc. Var. Partial Differential Equations, 22 (2005), pp.~129--172.
			
			\bibitem{GidRiv}
			{\sc P.~Gidoni and F.~Riva}, {\em A vanishing inertia analysis for finite dimensional rate--independent systems with non--autonomous dissipation, 
				and an application to soft crawlers}, Preprint (2020), arXiv:2007.09069.	
			
			\bibitem{Griffith}
			{\sc A.~Griffith}, {\em The phenomena of rupture and flow in solids},  Philos. Trans. Roy. Soc. London Ser. A, 221 (1920), pp.~163--198.	
			
			\bibitem{HeiMielk}
			{\sc M.~Heida and  A.~Mielke}, {\em Averaging of time--periodic dissipation potentials in rate--independent processes}, 
			Discrete and Continuous Dynamical Systems. Series S, 10.6 (2017), pp.~1303--1327.
			
			\bibitem{Iurl}
			{\sc F.~Iurlano}, {\em Fracture and plastic models as $\Gamma$-limits of damage models under different regimes},  
			Advances in Calculus of Variations, 6 (2013), pp.~165--189.	
			
			\bibitem{Kuhn}
			{\sc C.~Kuhn, A.~Schl\"uter and R.~M\"uller}, {\em On degradation functions in phase field fracture models},  Computational Material Science, 108 (2015), pp.~374--384.
			
			\bibitem{MielkRoub}
			{\sc A.~Mielke and T.~Roub\'i\v cek}, {\em Rate--independent damage processes in nonlinear elasticity},  
			Math. Models and Methods in Applied Science, 16, No. 2 (2006), pp.~177--209.
			
			\bibitem{MielkRoubbook}
			{\sc A.~Mielke and T.~Roub\'i\v cek}, {\em Rate--independent Systems: Theory and Application},  Springer-Verlag New York, (2015).
			
			\bibitem{Thom}
			{\sc  A.~Mielke and M.~Thomas}, {\em Damage of nonlinearly elastic materials at small strain – Existence and regularity results -}, Z. Angew. Math. Mech., 90 (2010), pp.~88–112.
			
			\bibitem{Nejar}
			{\sc B.~Nedjar}, {\em Elasto--plastic damage modelling including the gradient of damage: formulation and computational aspects}, 
			International J. of Solids and Structures, 38 (2001), pp.~5421--5451.
			
			\bibitem{Negri}
			{\sc M.~Negri}, {\em A quasi-static model for craquelure patterns}, 
			Mathematical Modeling in Cultural Heritage, MACH2019, Springer INdAM Series, 41 (2021), pp.~147--164.
			
			\bibitem{NegSca}
			{\sc M.~Negri and R.~Scala}, {\em A quasi-static evolution generated by local energy minimizers for an elastic material with a cohesive interface}, 
			Nonlinear Anal. Real World Appl., 38 (2017), pp.~271--305.
			
			\bibitem{NegVit}
			{\sc M.~Negri and E.~Vitali}, {\em Approximation and characterization of quasi-static $H^1$-evolutions for a cohesive interface with different loading-unloading regimes}, 
			Interfaces Free Bound., 20 (2018), pp.~25--67.
			
			\bibitem{PhamMarigo}
			{\sc K.~Pham, J.--J.~Marigo and  C.~Maurini}, {\em The issues of the uniqueness and the stability of the homogeneous response in uniaxial tests with gradient damage models},  
			J. of the Mechanics and Physics of Solids, 59 (2011), pp.~1163--1190.
			
			\bibitem{Qin}
			{\sc Z.~Qin, N.~M.~Pugno and M.~J.~Buehler}, {\em Mechanics of fragmentation of crocodile skin and other thin films}, Sci. Rep., 4 (2014), pp.~1–-7.	
			
			\bibitem{Rivquas}
			{\sc F.~Riva}, {\em On the approximation of quasistatic evolutions for the debonding of a thin film via vanishing inertia and viscosity}, 
			J. of Nonlinear Science, 30 (2020), pp.~903--951.
			
			\bibitem{Wu}
			{\sc J.--Y.~Wu}, {\em A unified phase-field theory for the mechanics of damage and quasi-brittle failure},  J. of the Mechanics and Physics of Solids, 103 (2017), pp.~72--99.
		}
		
	\end{thebibliography}
\end{document}